\documentclass{article}

\usepackage{amsthm,amssymb,latexsym,amsmath, fancyhdr}

\newcommand{\version}{Version 3.2,\ \  August 05, 2013}
\newcommand{\R}{\mathbb{R}} \newcommand{\C}{\mathbb{C}}
\newcommand{\Z}{\mathbb{Z}} 
\newcommand{\p}[1]{{\mathbb{P}^{#1}}} 
\newcommand{\op}[1]{{\cal O}_{\C\mathbb{P}^{#1}}}

\newcommand{\im}{{\rm Im}~}

\newcommand{\calm}{{\cal M}} 
\newcommand{\calo}{{\cal O}} \newcommand{\calf}{\mathcal{F}}

\newcommand{\bOmega}{{\boldsymbol{\boldsymbol{\Omega}}}}
\newcommand{\bMu}{{\boldsymbol{\boldsymbol \mu}}}
\newcommand{\4}{{/\!\!/\!\!/\!\!/}} 
\newcommand{\3}{{/\!\!/\!\!/}} 
\newcommand{\2}{{/\!\!/}}
\newcommand{\1}{{\sqrt{-1}}}
\newcommand{\cntrct}                % contraction with a vector field
{\hspace{2pt}\raisebox{1pt}{\text{$\lrcorner$}}\hspace{2pt}}
\renewcommand{\leq}{\leqslant}

\newtheorem{theorem}{Theorem}
\newtheorem{proposition}[theorem]{Proposition}
\newtheorem{lemma}[theorem]{Lemma}
\newtheorem{corollary}[theorem]{Corollary}
\newtheorem{remark}[theorem]{Remark}
\newtheorem{example}[theorem]{Example}
\newtheorem{definition}[theorem]{Definition}

\newtheorem{claim}[theorem]{Claim}

\makeatletter

\@addtoreset{equation}{section}
\@addtoreset{theorem}{section}
\makeatother

\def\Definition{\begin{definition}}
\def\ed{\end{definition}}
\def\Remark{\begin{remark}}
\def\er{\end{remark}}
\def\Claim{\begin{claim}}
\def\ec{\end{claim}}
\def\Corollary{\begin{corollary}}
\def\eco{\end{corollary}}
\def\Lemma{\begin{lemma}}
\def\el{\end{lemma}}
\def\Theorem{\begin{theorem}}
\def\et{\end{theorem}}
\def\Proposition{\begin{proposition}}
\def\ep{\end{proposition}}

\def\goth{\mathfrak}
\newcommand{\g}{{\goth g}} \newcommand{\gu}{{\goth u}} \newcommand{\gl}{{\goth g}{\goth l}}
\def\6{\partial}
\def\endproof{\hbox{\vrule width 4pt height 4pt depth 0pt}}
\newcommand{\restrict}[1]{{\left|_{{\phantom{|}\!\!}_{#1}}\right.}}
\newcommand{\arrow}{{\:\longrightarrow\:}}
\newcommand{\Tw}{\operatorname{Tw}}
\newcommand{\Ann}{\operatorname{Ann}}
\newcommand{\Alt}{\operatorname{Alt}}
\newcommand{\Av}{\operatorname{Av}}
\newcommand{\Mat}{\operatorname{Mat}}
\newcommand{\Id}{{\mathbf 1}}
\newcommand{\IdId}{\operatorname{Id}}
\newcommand{\Tr}{\operatorname{Tr}}
\renewcommand{\Re}{\operatorname{Re}}
\newcommand{\ev}{\operatorname{\sf ev}}
\newcommand{\St}{\operatorname{\sf St}}
\newcommand{\Cl}{\operatorname{{\cal C}l}}
\newcommand{\Sp}{\operatorname{Sp}}
\newcommand{\rk}{\operatorname{\sf rk}}
\newcommand{\Sec}{{\operatorname{Sec}}}
\newcommand{\Char}{{\operatorname{\sf char}}}
\newcommand{\Area}{{\operatorname{Area}}}
\newcommand{\Sym}{\operatorname{Sym}}
\newcommand{\End}{\operatorname{End}}

\newcommand{\Vol}{\operatorname{Vol}}
\newcommand{\Lie}{\operatorname{Lie}}
\newcommand{\Null}{\operatorname{Null}}

\makeatletter
\def\x@arrow{\DOTSB\Relbar}
\def\xlongrightarrowfill@{\arrowfill@\relbar\relbar\longrightarrow}
\newcommand{\xlongrightarrow}[2][]{%
	\ext@arrow 0099\xlongrightarrowfill@{#1}{#2}}
\makeatother

\begin{document}

\title{Trihyperk\"ahler reduction \\ and instanton bundles on $\C\p3$}
\author{Marcos Jardim \\ IMECC - UNICAMP \\
Departamento de Matem\'atica \\ Rua S\'ergio Buarque de Holanda, 651 \\
13083-859 Campinas, SP, Brazil  \\[4mm] Misha Verbitsky \\  
Laboratory of Algebraic Geometry, \\
Faculty of Mathematics, NRU HSE,\\
7 Vavilova Str. Moscow, Russia, {\em and}\\
Institute for the Physics and Mathematics of the Universe,\\
University of Tokyo,
    5-1-5 Kashiwanoha, Kashiwa, 277-8583, Japan
} 

\maketitle

\begin{abstract}
A trisymplectic structure on a complex $2n$-manifold is a 
3-dimensional space $\Omega$ of closed holomorphic 
forms such that any element of $\Omega$ has constant rank $2n$, $n$ or 0,
and degenerate forms in $\Omega$ belong to a non-degenerate quadric
hypersurface. We show that a trisymplectic manifold is equipped
with a holomorphic 3-web and the Chern
connection of this 3-web is holomorphic, 
torsion-free, and preserves the three 
symplectic forms. We construct a trisymplectic
structure on the moduli of regular rational 
curves in the twistor space of a hyperk\"ahler
manifold, and define a trisymplectic reduction
of a trisymplectic manifold, which is a complexified
form of a hyperk\"ahler reduction. We prove that
the trisymplectic reduction in the space of regular
rational curves on the twistor space  
of a hyperk\"ahler manifold $M$ is compatible with the
hyperk\"ahler reduction on $M$. 
  
As an application of these geometric ideas, we consider
the ADHM construction of instantons and show that the
moduli space of rank $r$, charge $c$ framed instanton
bundles on $\C\mathbb{P}^{3}$ is a smooth 
trisymplectic manifold of complex dimension $4rc$. In
particular, it follows that the moduli space of rank $2$,
charge $c$ instanton bundles on $\C\mathbb{P}^{3}$ is a
smooth complex manifold dimension $8c-3$, thus
settling part of a 30-year old conjecture.

\end{abstract}

\tableofcontents

\hfill

\noindent{\bf Acknowledgments.} 
The first named author is partially supported by the CNPq
grant number 302477/2010-1 and the FAPESP grant number
2011/01071-3. He thanks Amar Henni and Renato 
Vidal Martins for several discussions related to instanton bundles. 
The second named author was partially supported by the 
FAPESP grant number 2009/12576-9, AG Laboratory SU-HSE, RF government 
grant, ag. 11.G34.31.0023, RFBR grant
 12-01-00944, NRU-HSE  Academic Fund Program in 2013-2014, research grant
12-01-0179, and Simons-IUM fellowship. We are grateful to
Alexander Tikhomirov for his insight and comments, and
to Andrei Soldatenkov for his suggestion to use
the Clifford algebra in our treatment of trisymplectic structures.

%---------------------------------------------------------------
%---------------------------------------------------------------

%%%%%%%%%%%%%%%%%%%%%%%%%%%%%%%%%%%%%%%%%%%%%%%%

\section{Introduction}

%%%%%%%%%%%%%%%%%%%%%%%%%%%%%%%%%%%%%%%%%%%%%%%%

%%%%%%%%%%%%%%%%%%%%%%%%%%%%%%%%%%%%%%%%%%%%%%%%
\subsection{An overview}
%%%%%%%%%%%%%%%%%%%%%%%%%%%%%%%%%%%%%%%%%%%%%%%%

In our previous paper \cite{_JV:Instantons_}, we
introduced the notion of holomorphic $SL(2)$-webs, and
argued that manifolds equipped with a holomorphic $SL(2)$-web structure
may be regarded as the complexification of hypercomplex manifolds. We showed
that manifolds $M$ carrying such structures have a canonical
holomorphic connection, called the \emph{Chern connection}, which
is torsion-free and has holonomy in $GL(n,\C)$, where
$\dim_\C M=2n$.

The main example of holomorphic $SL(2)$-webs are given by
twistor theory: given a hyperk\"ahler manifold $M$, then
the space of regular holomorphic sections of the twistor
fibration $\pi:\Tw(M)\to\C\p1$ is equipped with a holomorphic
$SL(2)$-web. We then exploited this fact and the
Atiyah--Drinfeld--Hitchin--Manin (ADHM) construction of
instantons to show that the moduli space of framed
instanton bundles on $\C\p3$ is a holomorphic $SL(2)$-web.

The present paper is a sequel to
\cite{_JV:Instantons_}. Here, we expand in both aspects
of our previous paper. On one hand, we describe a new
geometric structure on complex manifolds, called {\em a 
trisymplectic structure}. The trisymplectic structure
is an important special case of a holomorphic
$SL(2)$-web. For a trisymplectic structure, we define
reduction procedure, allowing us to define 
a trisymplectic quotient. Applying these new ideas to the ADHM
construction of instantons allows us to give a better
description of the moduli space of framed instanton
bundles on $\C\p3$, and to prove its smoothness and
connectness. This allows us to solve a part of a 30-year old
conjecture regarding the moduli space of rank $2$
instanton bundles on $\C\p3$.

To be more precise, we begin by introducing the notion of
\emph{trisymplectic structures on complex manifolds} (see
Definition \ref{_trisymple_mani_Definition_} below), and
show that trisymplectic manifolds carry an induced
holomorphic $SL(2)$-web. Our first main goal is to
introduce the notion of a \emph{trisymplectic quotient} of
a trisymplectic manifold, which would enable us to
construct new examples of trisymplectic manifolds out of
known ones, e.g. flat ones. 

Next, we introduce the notion of \emph{trihyperk\"ahler
quotient} $\Sec_0(M)\4 G$ for a hyperk\"ahler manifold
$M$, equipped with an action of a Lie group $G$ by considering
the trisymplectic quotient of the space $\Sec_0(M)$ of regular
holomorphic sections of the twistor fibration of $M$.

Our first main result (Theorem
\ref{_trihk_red_equal_on_sec_Theorem_}) 
is compatibility between this procedure
and the hyperk\"ahler quotient, which we denote
by $M \3 G$. We show that,
under some reasonable conditions,  
the trihyperk\"ahler reduction $\Sec_0(M)\4 G$ 
admits an open embedding to the space $\Sec_0(M\3 G)$ of 
regular sections of the twistor fibration of 
the hyperk\"ahler quotient $M \3 G$.
This shows, in particular, that (similarly to the smoothness of
the hyperk\"ahler reduction) the trihyperk\"ahler
reduction of $M$ is a smooth trisymplectic manifold.

Our second main result provides an affirmative answer
to a long standing conjecture regarding the smoothness
and dimension of the moduli space of rank
$2$ instanton bundles on $\C\p3$, a.k.a mathematical instanton bundles
(see Section \ref{_appl_} for precise definitions).
More precisely, the moduli space of mathematical instanton bundles with
second Chern class (or \emph{charge}) $c$ is conjectured 
to be an irreducible, nonsingular
quasi-projective variety  of dimension $8c-3$ (c.f. \cite[Conjecture  1.2]{CTT}).
The truth of the conjecture for $c\le 5$ was established by various authors in the
past four decades: Barth settled the $c=1$ case in 1977
\cite{B1}; Hartshorne established the case $c=2$ in 1978
\cite{H}; Ellingsrud and Stromme settled the $c=3$ case in
1981 \cite{ES}; the irreducibility of the $c=4$ case was
proved by Barth in 1981 \cite{B2}, while the smoothness is
due to Le Potier \cite{LeP} (1983);  and
Coanda--Tikhomirov--Trautmann (2003). More recently,
Tikhomirov has shown in \cite{T} that irreducibility holds
for odd values of $c$. 

In the present paper, we apply the geometric techniques established
above to the ADHM construction of instantons, and show that the moduli 
space of rank $r$, charge $c$ \emph{framed} instanton
bundles on $\C\mathbb{P}^{3}$ is a smooth,
trisymplectic manifold of complex dimension $4rc$ (see
Theorem \ref{_instanton bdls_} below). It then follows easily (see Section
\ref{instantons_bundles} for the details) that the moduli space of mathematical instanton bundles of charge $c$ is a smooth complex manifold of dimension $8c-3$, thus settling
the smoothness part of the conjecture for all values of $c$.

%Together with a recent result of
%Markushevich and Tikhomirov \cite[Theorem 1.1]{MT}, one
%concludes that the moduli space of rank $2$ instanton
%bundles of charge $c$ on $\C\p3$ is a  rational,
%non-singular, quasi-projective variety of dimension $8c-3$. 

%%%%%%%%%%%%%%%%%%%%%%%%%%%%%%%%%%%%%%%%%%%%%%%%%%%%%%%%%%%%
\subsection{3-webs, $SL(2)$-webs and trisymplectic structures}
%%%%%%%%%%%%%%%%%%%%%%%%%%%%%%%%%%%%%%%%%%%%%%%%%%%%%%%%%%%%

Let $M$ be a real analytic manifold equipped with 
an atlas $\{U_i\hookrightarrow \R^n\}$ and 
real analytic transition functions
$\psi_{ij}$. A \emph{complexification} of $M$
(see \cite{_Grauert:real-ana_}) 
is a germ of a complex manifold, covered by
open sets $\{V_i\hookrightarrow \C^n\}$ indexed by the same set
as $\{U_i\}$, and with the transition function $\psi_{ij}^\C$
obtained by analytic extension of $\psi_{ij}$
into the complex domain. 

Complexification can be applied to a complex
manifold, by considering it as a real analytic
manifold first. As shown by Kaledin and Feix
(see \cite{_Feix_}, \cite{_Kaledin:book_} and the argument in
\cite[Section 1]{_JV:Instantons_}),
a complexification of a real analytic K\"ahler manifold
naturally gives a germ of a hyperk\"ahler manifold.
In the paper \cite{_JV:Instantons_}
we took the next step by looking at a 
complexification of a hyperk\"ahler manifold.
We have shown that such a complexification
is equipped with an interesting geometric
structure which we called {\em a holomorphic
$SL(2)$-web}. 

A holomorphic $SL(2)$-web on a complex manifold $M$
is a collection of involutive holomorphic sub-bundles
$S_t\subset TM$, $\rk S_t = \frac 1 2 \dim M$,
parametrized by $t\in \C\p1$, and satisfying
the following two conditions. First, $S_t \cap S_{t'}=0$
for $t \neq t'$, and second, the projector operators
$\Pi_{t, t'}$ of $TM$ onto $S_{t'}$ along $S_t$
generate an algebra isomorphic to the the algebra $\Mat(2)$ of
$2\times2$ complex matrices
 (cf. Definition \ref{_SL(2)_web_Definition_} and
Section \ref{_Mat_2_Subsection_} below).

This structure is a special case of a notion
of 3-web developed in 1930-ies by Blaschke and
Chern. Let $M$ be an even-dimensional manifold,
and $S_1, S_2, S_3$ a triple of pairwise
non-intersecting involutive sub-bundles of
$TM$ of dimension $\frac 1 2 \dim M$. Then 
$S_1, S_2, S_3$ is called {\bf a 3-web.}
Any 3-web on $M$ gives a rise to a natural connection
on $TM$, called {\bf a Chern connection}.
A Chern connection is one which preserves $S_i$, 
and its torsion vanishes on $S_1 \otimes S_2$;
such a connection exists, and is unique. 

Let $a,b,c\in \C \p1$ be three distinct points.
For any $SL(2)$-web, $S_a, S_b, S_c$ is clearly a 3-web.
In \cite{_JV:Instantons_} we proved that the
corresponding Chern connection is torsion-free
and holomorphic; also, it is independent from 
the choice of $a,b,c\in \C \p1$.
We also characterized such connections
in terms of holonomy, and characterized
an $SL(2)$-web in terms of a connection
with prescribed holonomy.

Furthermore, we constructed an $SL(2)$-web structure
on a component of the moduli space of rational curves 
on a twistor space of a hyperk\"ahler manifold.
By interpreting the moduli space of framed instanton bundles on $\C\p3$ in terms of rational
curves on the twistor space of the moduli space of framed bundles on $\C\p2$, we obtained a $SL(2)$-web on the smooth part of  the moduli space of framed instanton bundles on $\C\p3$.

In the present paper we explore this notion further,
studying those $SL(2)$-webs which appear as moduli spaces of 
rational lines in the twistor space of a hyperk\"ahler manifold. 

It turns out that (in addition to the $SL(2)$-web structure),
this space is equipped with the so-called {\em
trisymplectic structure} (see also Definition
\ref{_trisymple_mani_Definition_}).

\Definition \label{_trisy_intro_Definition_}
A {\bf weakly trisymplectic structure} on
a complex manifold $M$ is a 3-dimensional
subspace $\bOmega$ of $\Omega^2 M$ generated by a triple 
of holomorphic symplectic forms $\Omega_1, \Omega_2,
\Omega_3$, such that any linear combination
of $\Omega_1, \Omega_2, \Omega_3$
has rank $n=\dim M$, $\frac 1 2 n$, or 0.
If the set of degenerate forms in $\bOmega$
belongs to a non-degenerate quadric, $\bOmega$ is called
{\bf a trisymplectic structure}, and 
$(M,\bOmega)$ {\bf a trisymplectic manifold}.
\ed

In differential 
geometry, similar structures known as {\em 
hypersymplectic structures} were
studied by Arnol$'$d, Atiyah, Hitchin
and others (see e.g. \cite{_Arnold:hypersy_}).
The hypersymplectic manifolds are similar to hyperk\"ahler,
but instead of quaternions one deals with
an algebra $\Mat(2, \R)$ of split
quaternions. As one passes to complex manifolds
and complex-valued holomorphic symplectic
forms, the distinction between quaternions
and split quaternions becomes irrelevant.
Therefore, trisymplectic structures
serve as complexifications of 
both hypersymplectic and hyperk\"ahler
structures.

Consider a trisymplectic manifold
$(M, \Omega_1, \Omega_2, \Omega_3)$.
In Theorem \ref{_trisy_web_Theorem_} we show that
the set of degenerate linear combinations
of $\Omega_i$ is parametrized by
$\C\p1$ (up to a constant), and the
null-spaces of these 2-forms form
an $SL(2)$-web. We also prove that
the Chern connection associated with this
 $SL(2)$-web preserves the 2-forms $\Omega_i$
(Theorem \ref{_trisy_Chern_Theorem_}).
This allows one to characterize trisymplectic manifolds
in terms of the holonomy, similarly as it is done
in \cite{_JV:Instantons_} with $SL(2)$-webs.

%%%%%%%%%%%%%%%%%%%%%%%%%%%%%%%%%%%%%%%%%%%%%%%%
\Claim
Let $M$ be a complex manifold. Then there is a bijective
correspondence between trisymplectic structures on $M$,
and holomorphic connections with holonomy
which lies in $G=Sp(n, \C)$ acting
on $\C^{2n}\otimes_\C \C^2$ 
trivially on the second tensor multiplier
and in the usual way on $\C^{2n}$.
\ec

{\bf Proof:} Follows immediately from
Theorem \ref{_trisy_Chern_Theorem_}. 
\endproof

%%%%%%%%%%%%%%%%%%%%%%%%%%%%%%%%%%%%%%%%%%%%%%%%%%%%%%%%%%%%
\subsection{Trisymplectic reduction}
%%%%%%%%%%%%%%%%%%%%%%%%%%%%%%%%%%%%%%%%%%%%%%%%%%%%%%%%%%%%

In complex geometry, the symplectic reduction
is understood as a way of constructing the GIT
quotient geometrically.
Consider a K\"ahler manifold $M$
equipped with an action of a  compact
Lie group $G$. Assume that $G$ acts by
holomorphic isometries, and admits an equivariant
moment map $M \stackrel \mu \arrow \g^*$,
where $\g^*$ is the dual of the Lie algebra 
of $G$. The {\bf symplectic reduction}
$M \2 G$ is the quotient of $\mu^{-1}(0)$
by $G$. This quotient is a complex variety,
K\"ahler outside of its singular points.
When $M$ is projective, one can identify 
$M \2 G$ with the GIT quotient of $M$ 
by the action of the complexified Lie group $G_{\C}$. 
For more details on GIT and its relation
to the symplectic quotient, please see
\cite{_Mumfort_Fogarty_Kirwan_}.

A hyperk\"ahler quotient is defined
in a similar way. Recall that a hyperk\"ahler manifold
is a Riemannian manifold equipped with a triple
of complex structures $I,J,K$ which are K\"ahler
and satisfy the
quaternionic relations. Suppose that
a compact Lie group acts on $(M,g)$ by 
isometries which are holomorphic
with respect to $I,J,K$; such maps
are called \emph{hyperk\"ahler isometries}. Suppose,
moreover, that there exists a triple
of moment maps $\mu_I, \mu_J, \mu_K:\; M \arrow \g^*$
associated with the symplectic forms
$\omega_I, \omega_J, \omega_K$
constructed from $g$ and $I, J, K$.
The \emph{hyperk\"ahler quotient} (\cite{_HKLR_})
$M \3 G$ is defined as 
$\big(\mu^{-1}_I(0)\cap \mu^{-1}_J(0)\cap \mu^{-1}_K(0)\big)/G$.
Similarly to the K\"ahler case, this quotient
is known to be hyperk\"ahler outside of the singular locus.

This result is easy to explain if one looks
at the 2-form $\Omega:= \omega_J + \1 \omega_K$.
This form is holomorphically symplectic
on $(M,I)$. Then the complex 
moment map $\mu_\C:= \mu_J + \1 \mu_K$
is holomorphic on $(M,I)$, and the
quotient $M\3 G := \mu_\C^{-1}\2 G_\C$
is a K\"ahler manifold.
Starting from $J$ and $K$ instead of $I$,
we construct other complex structures
on $M \3 G$; an easy linear-algebraic
argument is applied to show that these
three complex structures satisfy
the quaternionic relations. 

Carrying this argument a step farther,
we repeat it for trisymplectic manifolds,
as follows. Let $(M, \Omega_1, \Omega_2, \Omega_3)$
be a trisymplectic manifold, that is, a complex
manifold equipped with a triple of holomorphic
symplectic forms satisfying the rank conditions
of Definition \ref{_trisy_intro_Definition_},
and $G_\C$ a complex Lie group acting on $M$
by biholomorphisms preserving $\Omega_1, \Omega_2,
\Omega_3$. Let $\mu_1$, $\mu_2$, and $\mu_3$ be the triple of
holomorphic moment maps, associated to $\Omega_1, \Omega_2,
\Omega_3$, which 
are assumed to be equivariant. \footnote{Equivariance of a moment
map is sometimes assumed in the definition, but we consider
it as an additional constraint; please see 
Subsection \ref{_trisymple_red_Subsection_} for 
more details and a precisely worded definition.}

The \emph{trisymplectic reduction}
is the quotient of $\mu^{-1}_1(0)\cap \mu^{-1}_2(0)\cap \mu^{-1}_3(0)$
by $G_\C$. 

Under some reasonable non-degeneracy assumptions,
we can show that a trisymplectic quotient
is also a trisymplectic manifold (Theorem
\ref{_trisy_quotient_trisy_Theorem_}). 

Notice that since $G_\C$ is non-compact,
this quotient is not always well-defined. To rectify
this, a trisymplectic version of GIT quotient
is proposed (Subsection \ref{_trihy_red_def_Subsection_}), 
under the name of \emph{trihyperk\"ahler reduction}.

%%%%%%%%%%%%%%%%%%%%%%%%%%%%%%%%%%%%%%%%%%%%%%%%%%%%%%%%%%%%
\subsection{Trihyperk\"ahler reduction}
%%%%%%%%%%%%%%%%%%%%%%%%%%%%%%%%%%%%%%%%%%%%%%%%%%%%%%%%%%%%

Let $M$ be a hyperk\"ahler manifold, and 
$\Tw(M) \stackrel \pi\arrow \C \p1$ its twistor space
(Subsection \ref{_hk_Subsection_}). 
A holomorphic section of $\pi$ is called
\emph{regular} if the normal bundle 
to its image is isomorphic to a sum
of $\dim M$ copies of $\calo(1)$.
Denote by $\Sec_0(M)$ the space of 
regular sections of $\pi$ (cf. Definition
\ref{_regu_sec_Definition_}). 

One may think of $\Sec_0(M)$ as of a complexification of a
hyperk\"ahler manifold $M$. It is the main example of a
trisymplectic manifold used in this paper. 

The trisymplectic structure on $\Sec_0(M)$ is easy to obtain
explicitly. Let $L$ be a complex structure
on $M$ induced by the quaternions (Subsection 
\ref{_hk_Subsection_}), and $\Omega_L$
the corresponding holomorphic symplectic form on $(M, L)$.
Let \[ \ev_L:\; \Sec_0(M) \arrow (M, L)\]
be the {\em evaluation map} sending a twistor
section $s:\; \C \p1 \arrow \Tw(M)$ to 
$s(L)$ (we use the standard identification of
the space of induced complex structures with 
$\C \p1$). Let $\bOmega$ be the 3-dimensional space
of holomorphic 2-forms on $\Sec_0(M)$ generated 
by $\ev_I^*(\Omega_I)$, $\ev_J^*(\Omega_J)$
and $\ev_K^*(\Omega_K)$. Then $\bOmega$
defines a trisymplectic structure 
(Claim \ref{_trisy_on_Sec_Claim_}).

In this particular situation,
the trisymplectic quotient
can be defined using a GIT-like
construction as follows. 

Let $G$ be a compact Lie group
acting on a hyperk\"ahler manifold
$M$ by hyperk\"ahler isometries.
Then $G$ acts on $\Sec_0(M)$
preserving the trisymplectic
structure described above. Moreover, there is a natural
K\"ahler metric on $\Sec_0(M)$
constructed in \cite{_NHYM_} 
as follows. The twistor space
$\Tw(M)$ is naturally isomorphic,
as a smooth manifold, to $M \times \C\p1$.
Consider the product metric on $\Tw(M)$, and
let $\nu:\; \Sec_0(M)\arrow \R^{+}$
be a map associating to a complex curve
its total Riemannian volume.  In 
\cite{_NHYM_} it was shown that
$\nu$ is a K\"ahler potential,
that is, $dd^c \nu$ is a K\"ahler
form on $\Sec_0(M)$.

Let $\bOmega$ be the standard 3-dimensional space of holomorphic
2-forms on \\ $\Sec_0(M)$, 
\[ \bOmega=\langle \ev_I^*(\Omega_I), 
\ev_J^*(\Omega_J), \ev_K^*(\Omega_K)\rangle.
\]
Then the corresponding triple of holomorphic moment
maps is generated by $\mu^\C_I\circ \ev_I$,
$\mu^\C_J\circ \ev_J$, $\mu^\C_K\circ \ev_K$,
where $\mu^\C_L$ is a holomorphic moment map
of $(M,L)$. This gives a description of 
the zero set of the trisymplectic moment
map $\bMu_\C:\; \Sec_0(M)\arrow \g^*\otimes_\R \C^3$,

As follows from 
Proposition \ref{_trisymple_on_Sec_Proposition_},
a rational curve $s\in \Sec_0(M)$
lies in $\bMu^{-1}_\C(0)$ if and only if
$s$ lies in a set of all pairs
$(m,t)\in M\times\C\p1 \simeq \Tw(M)$
satisfying $\mu^\C_t(m)=0$, where
$\mu^\C_t:\; (M, t) \arrow \g^*\otimes_\R \C$
is the holomorphic moment map corresponding to the
complex structure $t$.

Now, the zero set  $\bMu^{-1}_\C(0)$ of the trisymplectic
moment map is a K\"ahler
manifold, with the K\"ahler metric $dd^c\nu$
defined as above. Therefore, one could define
the symplectic quotient $\bMu^{-1}_\C(0)\2 G$.
This quotient, denoted by $\Sec_0(M) \4 G$, 
is called the \emph{trihyperk\"ahler
quotient} of $\Sec_0(M)$ (see 
Definition \ref{_trihype_re_Definition_} for further details).

One of the main results of the present paper is the
following theorem relating the trihyperk\"ahler
quotient and the hyperk\"ahler quotient.

\Theorem\label{_trihy_hy_intro_Theorem_}
Let $M$ be flat hyperk\"ahler manifold, and $G$ a compact
Lie group acting on $M$ by hyperk\"ahler automorphisms.
Suppose that the hyperk\"ahler moment map exists, and the
hyperk\"ahler quotient $M\3 G$ is smooth.
Then there exists an open embedding
$\Sec_0(M)\4 G\stackrel \Psi\arrow \Sec_0(M\3 G)$,
which is compatible with the trisymplectic
structures on $\Sec_0(M)\4 G$ and $\Sec_0(M\3 G)$.
\et

{\bf Proof:} This is Theorem \ref{_trihk_red_equal_on_sec_Theorem_}.
\endproof

\hfill

The flatness of $M$, assumed in Theorem 
\ref{_trihy_hy_intro_Theorem_}, does not seem to be necessary, but
we were unable to prove it without this assumption. 

%%%%%%%%%%%%%%%%%%%%%%%%%%%%%%%%%%%%%%%%%%%%%%%%%%%%%%%%%%%%
\subsection{Framed instanton bundles on $\C\p3$}
%%%%%%%%%%%%%%%%%%%%%%%%%%%%%%%%%%%%%%%%%%%%%%%%%%%%%%%%%%%%

In Section \ref{_appl_}, the geometric techniques introduced in the previous
sections are applied to the study of the moduli space of framed instanton bundles
on $\C\p3$.

Recall that a holomorphic vector bundle $E\to\C\p3$ is called an {\em instanton bundle}
if $c_1(E)=0$ and $H^0(E(-1))=H^1(E(-2))=H^{2}(E(-2))=H^3(E(-3))=0$. The integer $c:=c_2(E)$ is called the {\em charge} of $E$.

This nomenclature comes from the fact that instanton bundles which are trivial on the lines of the twistor fibration $\C\p3\to S^4$ (a.k.a. \emph{real lines}) are in 1-1 correspondence, via twistor transform, with non-Hermitian anti-self-dual connections on $S^4$ (see  \cite[Section 3]{_JV:Instantons_}). Note however that there are instanton bundles which are not trivial on every real line.

Moreover, given a line $\ell\subset\p3$, a {\em framing}
on $E$ at $\ell$ is the choice of an isomorphism
$\phi:E|_\ell\to{\cal O}_{\ell}^{\oplus{\rm rk}E}$. A {\em
  framed instanton bundle} is a pair $(E,\phi)$ consisting
of an instanton bundle $E$ restricting trivially to $\ell$
and a framing $\phi$ at $\ell$. Two framed bundles $(E,\phi)$
and $(E',\phi')$ are
isomorphic if there exists a bundle isomorphism $\Psi:E\to
E'$ such that $\phi'=\phi\circ(\Psi|_\ell)$. 

Frenkel and the first named author established in \cite{FJ2} a 1-1 correspondence between isomorphism classes of framed instanton bundles on $\C\p3$ and solutions of the \emph{complex ADHM equations} (a.k.a \emph{1-dimensional ADHM equation}) in \cite{J-cr}). 

More precisely, let $V$ and $W$ be complex vector spaces of dimension $c$ and $r$, respectively, and consider matrices ($k=1,2$) $A_k,B_k\in{\rm End}(V)$,
$I_k\in{\rm Hom}(W,V)$ and $J_k\in{\rm Hom}(V,W)$. The 1-dimensional ADHM equations are
\begin{equation}\label{um}
\left\{ \begin{array}{l}
~[ A_1 , B_1 ]+I_1J_1 = 0 \\
~[ A_2 , B_2 ]+I_2J_2 = 0 \\
~[ A_1 , B_2 ] + [ A_2,B_1 ] + I_1J_2 + I_2J_1 = 0
\end{array} \right.
\end{equation}
One can show  \cite[Main Theorem]{FJ2} that the moduli space of framed instanton bundles on $\C\p3$ coincides with the set of \emph{globally regular} solutions (see Definition \ref{gr-def} below) of the 1-dimensional ADHM equations modulo the action of $GL(V)$.

It turns out that the three equations in (\ref{um}) are precisely the three components of a trisymplectic moment map
$\bMu_\C:\Sec_0(M)\to\gu(V)^*\otimes_\R\Gamma(\op1(2))$
on (an open subset of) a flat hyperk\"ahler manifold $M$, so that the moduli space of framed instanton bundles coincides with a trihyperk\"ahler reduction of a flat space (Theorem \ref{_insta_same_as_trihk_Theorem_}). It then follows that the moduli space of framed 
instanton bundles on $\C \p3$ of rank $r$ and charge $c$ is a smooth trisymplectic manifold of dimension $4rc$.

On the other hand, a \emph{mathematical instanton bundle} is a rank $2$ stable holomorphic vector bundle $E\to\C\p3$ with $c_1(E)=0$ and $H^1(E(-2))=0$. It is easy to see, using Serre duality and stability, that every mathematical instanton bundle is a rank $2$ instanton bundle. Conversely, every rank $2$ instanton bundle is stable, and thus a mathematical instanton bundle. We explore this fact to complete the paper in Section \ref{instantons_bundles} by showing how the smoothness of the moduli space of framed rank $2$ instanton bundles 
settles the smoothness part of a conjecture on the moduli space of mathematical instanton bundles; for its precise formulation, see \cite[Conjecture 1.2]{CTT}.

%%%%%%%%%%%%%%%%%%%%%%%%%%%%%%%%%%%%%%%%%%%%%%%%%%%%%%%%%%%%

\section{$SL(2)$-webs on complex manifolds} \label{_3_webs_section_}

%%%%%%%%%%%%%%%%%%%%%%%%%%%%%%%%%%%%%%%%%%%%%%%%%%%%%%%%%%%%

In this section, we repeat basic results about
$SL(2)$-webs on complex manifolds. We follow \cite{_JV:Instantons_}.

%%%%%%%%%%%%%%%%%%%%%%%%%%%%%%%%%%%%%%%%%%%%%%%%%%%%%%%%%%%%
\subsection{$SL(2)$-webs and twistor sections}
\label{_3_webs_Subsection_}
%%%%%%%%%%%%%%%%%%%%%%%%%%%%%%%%%%%%%%%%%%%%%%%%%%%%%%%%%%%%

The following notion is based on
a classical notion of a 3-web, developed
in the 1930-ies by Blaschke and Chern, and much
studied since then.

\Definition\label{_SL(2)_web_Definition_}
Let $M$ be a complex manifold,
$\dim_\C M = 2n$, and $S_t \subset TM$
a family of $n$-dimensional holomorphic sub-bundles,
parametrized by $t\in \C\p1$.
This family is called {\bf a holomorphic $SL(2)$-web}
if the following conditions are satisfied
\begin{description}
\item[(i)] Each $S_t$ is involutive (integrable), that is, 
$[S_t, S_t] \subset  S_t$.
\item[(ii)] For any distinct points $t, t'\in \C\p1$,
the foliations $S_t$, $S_{t'}$ are transversal:
$S_t \cap S_{t'}=\emptyset$.
\item[(iii)] The projections $P_{t,t'}:\; TM \arrow S_t\hookrightarrow TM$
of $TM$ to $S_t$ along $S_{t'}$ generate a 4-dimensional
sub-bundle $\goth a$ within $\End(TM)$, which is closed under multiplication.
\item[(iv)] Each fiber of $\goth a$ is isomorphic to the algebra $\Mat(2)$
of 2-dimensional matrices. 
\end{description}
\ed

Since $S_t$ and $S_{t'}$ are 
mid-dimensional, transversal foliations, it follows that
$T_mM = S_t(m)\oplus S_{t'}(m)$ for each 
point $m\in M$. According to this splitting,
$P_{t,t'}(m)$ is simply a projection onto the first factor.

\Definition
(see e.g. \cite{_Atiyah:conne_})
Let $B$ be a holomorphic vector bundle
over a complex manifold $M$. A {\bf holomorphic connection}
on $B$ is a holomorphic differential operator
$\nabla:\; B \arrow B \otimes \Omega^1 M$
satisfying $\nabla(fb) = b \otimes df + f \nabla(b)$,
for any folomorphic function $f$ on $M$.
\ed

\Remark
Let $\nabla$ be a holomorphic connection on a holomorphic
bundle, considered as a map $\nabla:\; B \arrow B \otimes \Lambda^{1,0} M$,
and $\bar \6:\; B \arrow B \otimes \Lambda^{0,1} M$ the
holomorphic structure operator. The sum 
$\nabla_f:=\nabla+ \bar\6$ is clearly a connection.
Since $\nabla$ is holomorphic, $\nabla\bar\6 + \bar\6 \nabla=0$,
hence the curvature $\nabla_f^2$ is of type $(2,0)$.
The converse is also true: a $(1,0)$-part of a 
connection with curvature of type $(2,0)$
is always a holomorphic connection.
\er

%%%%%%%%%%%%%%%%%%%%%%%%%%%%%%%%%%%%%%%%%%%%%%%%
\Proposition\label{_torsion-free_Proposition_}
(\cite{_JV:Instantons_})
Let $S_t, t\in \C\p1$ be an $SL(2)$-web. 
Then there exists a unique torsion-free holomorphic
connection preserving $S_t$, for all $t\in \C \p1$. \endproof
\ep

%%%%%%%%%%%%%%%%%%%%%%%%%%%%%%%%%%%%%%%%%%%%%%%%
\Definition
This connection is called {\bf a Chern connection}
of an $SL(2)$-web.
\ed

%%%%%%%%%%%%%%%%%%%%%%%%%%%%%%%%%%%%%%%%%%%%%%%%
\Theorem
(\cite[Theorem 2.13]{_JV:Instantons_})
Let $M$ be a manifold equipped with a holomorphic 
$SL(2)$-web. Then its Chern connection is a torsion-free
affine holomorphic connection with holonomy in
$GL(n, \C)$ acting on $\C^{2n}$ as a centralizer of 
an $SL(2)$-action, where $\C^{2n}$ is a direct
sum of $n$ irreducible $GL(2)$-representations
of weight 1. Conversely, every connection with
such holonomy preserves a holomorphic $SL(2)$-web.
\endproof 
\et

%%%%%%%%%%%%%%%%%%%%%%%%%%%%%%%%%%%%%%%%%%%%%%%%
\subsection{Hyperk\"ahler manifolds}
\label{_hk_Subsection_}
%%%%%%%%%%%%%%%%%%%%%%%%%%%%%%%%%%%%%%%%%%%%%%%%

%%%%%%%%%%%%%%%%%%%%%%%%%%%%%%%%%%%%%%%%%%%%%%%%
\Definition
Let $(M,g)$ be a Riemannian manifold, and $I,J,K$
endomorphisms of the tangent bundle $TM$ satisfying the
quaternionic relations
\[
I^2=J^2=K^2=IJK=-\Id_{TM}.
\]
The triple $(I,J,K)$ together with
the metric $g$ is called {\bf a hyperk\"ahler structure}
if $I, J$ and $K$ are integrable and K\"ahler with respect to $g$.
\ed

Consider the K\"ahler forms $\omega_I, \omega_J, \omega_K$
on $M$:
\begin{equation}\label{_omega_I,J,K_defi_Equation_}
\omega_I(\cdot, \cdot):= g(\cdot, I\cdot), \ \
\omega_J(\cdot, \cdot):= g(\cdot, J\cdot), \ \
\omega_K(\cdot, \cdot):= g(\cdot, K\cdot).
\end{equation}
An elementary linear-algebraic calculation implies
that the 2-form 
\begin{equation}\label{_holo_symple_on_hk_Equation_}
\Omega:=\omega_J+\1\omega_K
\end{equation}
is of Hodge type $(2,0)$
on $(M,I)$. This form is clearly closed and
non-degenerate, hence it is a holomorphic symplectic form.

In algebraic geometry, the word ``hyperk\"ahler''
is essentially synonymous with ``holomorphically
symplectic'', due to the following theorem, which is
implied by Yau's solution of Calabi conjecture
(\cite{_Beauville_,_Besse:Einst_Manifo_}).

%%%%%%%%%%%%%%%%%%%%%%%%%%%%%%%%%%%%%%%%%%%%%%%%
\Theorem\label{_Calabi-Yau_Theorem_}
Let $M$ be a compact, K\"ahler, holomorphically
symplectic manifold, $\omega$ its K\"ahler form, $\dim_\C M =2n$.
Denote by $\Omega$ the holomorphic symplectic form on $M$.
Assume that $\int_M \omega^{2n}=\int_M (\Re\Omega)^{2n}$.
Then there exists a unique hyperk\"ahler metric $g$ within the same
K\"ahler class as $\omega$, and a unique hyperk\"ahler structure
$(I,J,K,g)$, with $\omega_J = \Re\Omega$, $\omega_K = \im\Omega$.
\endproof \et

Every hyperk\"ahler structure induces a whole 2-dimensional
sphere of complex structures on $M$, as follows. 
Consider a triple $a, b, c\in\R$, $a^2 + b^2+ c^2=1$,
and let $L:= aI + bJ +cK$ be the corresponging quaternion. 
Quaternionic relations imply immediately that $L^2=-1$,
hence $L$ is an almost complex structure. 
Since $I, J, K$ are K\"ahler, they are parallel with respect
to the Levi-Civita connection. Therefore, $L$ is also parallel.
Any parallel complex structure is integrable and K\"ahler.
Complex structures of the form $L= aI + bJ +cK$ are called the
\emph{complex structures induced by the hyperk\"ahler structure}.
The corresponding complex manifold is denoted by $(M,L)$.
There is a 2-dimensional holomorphic family of 
induced complex structures, and the total space
of this family is called the \emph{\bf twistor space}
of a hyperk\"ahler manifold; it is constructed as follows. 

Let $M$ be a hyperk\"ahler manifold. Consider the product  $\Tw(M) = M
\times S^2$. Embed the sphere $S^2 \subset {\mathbb H}$ into the quaternion algebra
${\mathbb H}$ as the subset of all quaternions $J$ with $J^2 = -1$. For every point
$x = m \times J \in X = M \times S^2$ the tangent space $T_x\Tw(M)$ is
canonically decomposed $T_xX = T_mM \oplus T_JS^2$. Identify $S^2$
with $\C \p1$, and let $I_J:T_JS^2 \to T_JS^2$ be the complex 
structure operator. Consider the complex structure $I_m:T_mM \to T_mM$ 
on $M$ induced by $J \in S^2 \subset {\mathbb H}$.

The operator $I_{\Tw} = I_m \oplus I_J:T_x\Tw(M) \to T_x\Tw(M)$ 
satisfies $I_{\Tw} \circ I_{\Tw} =
-1$. It depends smoothly on the point $x$, hence it defines 
an almost complex structure on $\Tw(M)$. This almost 
complex structure is known to be integrable
(see e.g. \cite{_Salamon_}). 

%%%%%%%%%%%%%%%%%%%%%%%%%%%%%%%%%%%%%%%%%%%%%%%%%%%
\Definition
The space $\Tw(M)$ constructed above is called {\bf the twistor space}
of the hyperk\"ahler manifold $M$.
\ed

%%%%%%%%%%%%%%%%%%%%%%%%%%%%%%%%%%%%%%%%%%%%%%%%
\subsection{An example: rational curves on a twistor
  space}
\label{_ratcurves_Subsection_}
%%%%%%%%%%%%%%%%%%%%%%%%%%%%%%%%%%%%%%%%%%%%%%%%

The basic example of holomorphic $SL(2)$-webs
comes from hyperk\"ahler geometry. Let $M$ be a hyperk\"ahler manifold, and 
$\Tw(M)$ its twistor space. Denote by $\Sec(M)$ the
space of holomorphic sections of the
twistor fibration ${\rm Tw}(M)\stackrel\pi\arrow\C\p1$, also known as \emph{twistor sections}.

We consider $\Sec(M)$ as a complex variety, with the
complex structure induced from the Douady space of rational
curves on $\Tw(M)$. Clearly, for any $C \in \Sec(M)$,
$T_C \Sec(M)$ is a subspace in the space of sections of 
the normal bundle $N_C$. This normal bundle is
naturally identified with $T_\pi\Tw(M)\restrict C$, where
$T_\pi\Tw(M)$ denotes the vertical tangent bundle.

For each point $m \in M$, one has a horizontal section
$C_m:=\{m\} \times \C\p1$ of $\pi$. The space of
horizontal sections of $\pi$ is denoted $\Sec_{hor}(M)$;
it is naturally identified with $M$. It is easy to check
that $N C_m= {\cal O}(1)^{\dim M}$,
hence some neighbourhood of $\Sec_{hor}(M)\subset 
\Sec(M)$ is a smooth manifold of
dimension $2\dim M$. It is easy to see
that $\Sec(M)$ is a complexification of
$M \simeq \Sec_{hor}(M)$, considered as a
real analytic manifold (see \cite{_Verbitsky:hypercomple_});
in fact, the real analytic structure on $M$ is constructed
by identifying a germ of $\Sec(M)$ with a complexification. 

%%%%%%%%%%%%%%%%%%%%%%%%%%%%%%%%%%%%%%%%%%%%%%%%%%%%%%%%%%%%
\Definition\label{_regu_sec_Definition_}
A twistor section $C\in \Sec(M)$ whose normal bundle $N_C$ is isomorphic to
${\cal O}(1)^{\dim M}$ is called {\bf regular}.
\ed

Let $\Sec_0(M)$ be the subset of $\Sec(M)$ consisting of regular twistor sections.
Clearly, $\Sec_0(M)$ is a smooth, Zariski open subvariety in $\Sec(M)$, containing the set $\Sec_{hor}(M)$ of horizontal twistor sections.

\Proposition\label{web_on_Sec}
The space $\Sec_0(M)$ of regular twistor sections admits the structure of a holomorphic $SL(2)$-web.
\ep

\begin{proof} A holomorphic $SL(2)$-web on $\Sec_0(M)$ can be constructed as follows. For each $C \in \Sec_0(M)$ and $t\in \C\p1=C$, define $S_t\subset TC= \Gamma_C (N_C)$ as the space of all sections of $N_C$ vanishing at $t\in C$.

It is not difficult to check that this is a holomorphic $SL(2)$-web. 
Transversality of $S_t$ and $S_{t'}$ follows easily from the fact that
a section of ${\cal O}(1)$ vanishing at two points must be zero.
Integrability of $S_t$ is also clear, since the leaves of $S_t$ are fibers of the evaluation map $ev_t:\; \Sec(M) \arrow \Tw(M)$, mapping $C: \; \C\p1
\arrow \Tw(M)$ to $C(t)$. 

Moreover, let $S$ be a complex vector space of dimension $\dim M$, so that $N_C\simeq S\otimes{\cal O}(1)$. Note that $\Gamma_C (N_C)\simeq S\otimes_\C \C^2$
and that the projection maps $P_{t, t'}$ act on $V \otimes_\C \C^2$ only through the second component; it is then easy to see that conditions (iii) and (iv) in Definition \ref{_SL(2)_web_Definition_} are also satisfied.
\end{proof}

\hfill

The space $\Sec_0(M)$ is the main example of an $SL(2)$-web
manifold we consider in this paper; the structure defined in the proof above is called the standard holomorphic $SL(2)$-web structure on $\Sec_0(M)$.

%%%%%%%%%%%%%%%%%%%%%%%%%%%%%%%%%%%%%%%%%%%%%%%%%%%%%%%%%%%%

\section{Trisymplectic structures on vector spaces}

%%%%%%%%%%%%%%%%%%%%%%%%%%%%%%%%%%%%%%%%%%%%%%%%%%%%%%%%%%%%

%%%%%%%%%%%%%%%%%%%%%%%%%%%%%%%%%%%%%%%%%%%%%%%%%%%%%%%%%%%%%
\subsection{Trisymplectic structures and $\Mat(2)$-action}
\label{_Mat_2_Subsection_}
%%%%%%%%%%%%%%%%%%%%%%%%%%%%%%%%%%%%%%%%%%%%%%%%%%%%%%%%%%%%%

This section is dedicated to the study of the following linear algebraic objects, which will be the basic ingredient in the new geometric structures we will introduce later.

%%%%%%%%%%%%%%%%%%%%%%%%%%%%%%%%%%%%%%%%%%%%%%%%%%%%
\Definition
Let $\bOmega$ be a 3-dimensional space of complex linear 2-forms
on an even dimensional complex vector space $V$. Assume that
\begin{description}
\item[(i)] $\bOmega$ contains a non-degenerate form;
\item[(ii)] For each non-zero degenerate $\Omega\in \bOmega$,
one has $\rk \Omega = \frac 1 2 \dim V$.
\end{description}
Then $\bOmega$ is called {\bf a weakly trisymplectic structure on $V$},
and $(V, \bOmega)$ {\bf a weakly trisymplectic space.}

The pair $(V, \bOmega)$ is called {\bf a trisymplectic space} if, in addition, the set $S$ of all degenerate forms $\Omega\in \bOmega$ lies in a non-degenerate
quadric $R\subset \bOmega$.
\ed

\Remark
If $V$ is not a complex, but a real  vector space, 
this notion defines either a quaternionic Hermitian structure,
or a structure known as {\bf hypersymplectic} and associated
with a action of split quaternions, c.f. \cite{_Arnold:hypersy_}.
\er

%%%%%%%%%%%%%%%%%%%%%%%%%%%%%%%%%%%%%%%%%%%%%%%%%%%%%%%%%%%%
\begin{example}\label{_Clifford_Example_}
Let $V=V_1\oplus V_1^*$ be a $2n$-dimensional
vector space ($n>1$) equipped with a scalar product $\langle \cdot, \cdot\rangle$
in a usual way,
$\Cl(V)=\Mat(W)$ its Clifford algebra, $h$ the
standard metric on the spinorial representation $W:=\Lambda^* V_1$, and
$V \hookrightarrow \goth{so}(W,h)$ the Clifford action. 
Identifying $\goth{so}(W)$ with $\Lambda^2 W$, we obtain 
a map $V \stackrel \phi \hookrightarrow \Lambda^2 W$. It is easy to see
that $\rk\phi(v)=\dim \Lambda^2 W$ unless $v\in V_1$ or
$v\in V_1^*$, and in the latter case $\rk\phi(v)=\frac 1 2 \dim \Lambda^2 W$.
If $L$ is a non-isotropic 3-dimensional subspace $V$, then the pair 
$(V,\phi(L))$ yields an example of a weakly trisymplectic space. Moreover, if 
the restriction $\langle \cdot, \cdot\rangle\restrict L$ is non-degenerate, then
the pair $(V,\phi(L))$ is a trisymplectic space.
\end{example}

%%%%%%%%%%%%%%%%%%%%%%%%%%%%%%%%%%%%%%%%%%%%%%%%%%%%
\Lemma\label{_annula_dont_interse_Lemma_}
Let $(V, \bOmega)$ be a be a weakly trisymplectic space,
and $\Omega_1, \Omega_2\in \bOmega$ two non-zero, 
degenerate forms which are not proportional.
Then the annihilator $\Ann(\Omega_1)$
does not intersect $\Ann(\Omega_2)$.
\el

{\bf Proof:} Indeed, if these two spaces intersect in a subspace 
$C\subset V$, strictly contained in $\Ann(\Omega_1)$,
some linear combination of $\Omega_1$
and $\Omega_2$ would have annihilator $C$,
which is impossible, because $0< \dim C < \frac 1 2 \dim V$.
If $\Ann(\Omega_1)= \Ann(\Omega_2)$, we could
consider $\Omega_1, \Omega_2$ as non-degenerate 
forms $\Omega_1\restrict W, \Omega_2\restrict W$ 
on $W:=V/\Ann(\Omega_2)$, which are obviously 
not proportional. We interpret $\Omega_i\restrict W$ as a bijective map
from $W$ to $W^*$. Let $w$ be an eigenvector of the operator
$\left(\Omega_2\restrict W\right)^{-1} \circ \Omega_1\restrict W\in \End(W)$, and 
$\lambda$ its eigenvalue. Then $\Omega_1(w, x) = \lambda \Omega_2(w,x)$,
for each $x\in W$, hence $w$ lies in the annihilator of
$\Omega_1\restrict W - \lambda\Omega_2\restrict W$. Then 
$\Omega_1 - \lambda\Omega_2$ has an annihilator strictly
larger than $\Ann(\Omega_2)$, which is impossible,
unless $\Omega_1 = \lambda\Omega_2$. \endproof

\hfill

Given two non-proportional, degenerate forms $\Omega_1,
\Omega_2\in \bOmega$, one has that
$V= \Ann(\Omega_1)\oplus \Ann(\Omega_2)$ by the previous
Lemma. Thus one can consider projection
operators $\Pi_{\Omega_1,\Omega_2}$ of $V$ onto
$\Ann(\Omega_1)$ along $\Ann(\Omega_2)$.
It turns out that the Clifford algebra action
used in Example \ref{_Clifford_Example_}
can be reconstructed from the trisymplectic
structure.  

%%%%%%%%%%%%%%%%%%%%%%%%%%%%%%%%%%%%%%%%%%%%%%%%%%%%
\Proposition\label{_trisymple_Cliff_Proposition_}
Let $(V, \bOmega)$ be a weakly 
trisymplectic vector space, and let 
$H\subset \End(V)$ be the subspace generated by
projections $\Pi_{\Omega_1,\Omega_2}$ for all pairs of non-proportional,
degenerate forms $\Omega_1, \Omega_2\in \bOmega$. Then 
 $H$ is a 4-dimensional subalgebra of $\End(V)$,
isomorphic to a Clifford algebra $\Cl(W)$, where $W$ is a 2-dimensional
subspace in $H$.
\ep

\begin{proof} 
{\bf Step 1:} We prove that the space $H\subset \End(V)$
is an algebra, and satisfies $\dim H\leq 4$.

Let $\Omega_1, \Omega_2 \in \bOmega$ be two forms
which are not proportional, and assume $\Omega_2$ is non-degenerate. 
Consider the operator $\phi_{\Omega_1, \Omega_2}\in \End(V)$, 
defined by $\phi_{\Omega_1, \Omega_2}:= \Omega_1 \circ \Omega_2^{-1}$,
where $\Omega_1, \Omega_2$ are understood as operators from $V$ to $V^*$.
As in the proof of Lemma \ref{_annula_dont_interse_Lemma_},
consider an eigenvector $v$ of $\phi_{\Omega_1, \Omega_2}$,
with the eigenvalue $\lambda$. Then $\Omega_1(v, x) = \lambda \Omega_2(v,x)$,
for each $x\in V$, hence $v$ lies in the annihilator of 
$\Omega:= \Omega_1 - \lambda\Omega_2$. Since $\Omega_i$ 
are non-proportional, $\Omega$ is non-zero, hence
$\rk \Omega = \frac 1 2 \dim V$. This implies that each
eigenspace of $\phi_{\Omega_1, \Omega_2}$ has dimension
$\frac 1 2 \dim V$. Choosing another eigenvalue $\lambda'$
and repeating this procedure, we obtain a $2$-form
$\Omega':= \Omega_1 - \lambda'\Omega_2$, also degenerate.
Let $S$ and $S'$ be the annihilators of $\Omega$ and $\Omega'$,
respectively. Let also $\Pi_{S,S'}$ and $\Pi_{S',S}$ be the projections of
$V$ onto $S$ or $S'$ along $S'$ or $S$.
It follows that 
\begin{equation}\label{_phi_through_projectors_Equation_} 
 \phi_{\Omega_1, \Omega_2}=\lambda\Pi_{S',S}+\lambda'\Pi_{S,S'}.
\end{equation}
and $\phi_{\Omega_1, \Omega_2}$ 
can be expressed in an appropriate basis by the matrix
\begin{equation}\label{_phi_matrix_Equation_}
\phi_{\Omega_1, \Omega_2}=\begin{pmatrix}
\lambda &0&0 &\hdotsfor{1} &0&0&0\\
0&\lambda &0 &\hdotsfor{1} &0&0&0\\
0&0&\lambda &\hdotsfor{1} &0&0&0\\
\vdots&\vdots&\vdots&
\ddots
&\vdots&\vdots&\vdots\\
0&0&0 &\hdotsfor{1} &\lambda'&0&0\\
0&0&0 &\hdotsfor{1} &0&\lambda'&0\\
0&0&0 &\hdotsfor{1} &0&0&\lambda'
\end{pmatrix}.
\end{equation}

{}From \eqref{_phi_through_projectors_Equation_} 
it is clear that the space $H$ is generated by all 
$\phi_{\Omega_1, \Omega_2}$. It is also clear that when $\Omega_2$ is also
non-degenerate, the operator $\phi_{\Omega_2, \Omega_1}$
can be expressed as a linear combination of 
$\phi_{\Omega_1, \Omega_2}$ and 
$\phi_{\Omega_1, \Omega_1}=\IdId_V$.

Since non-degenerate forms constitute 
a dense open subset of $\bOmega$, one can choose a basis
$\Omega_1, \Omega_2, \Omega_3$ consisting of non-degenerate forms. 
Since $\phi_{\Omega_i, \Omega}$ is expressed as
a linear combination of $\phi_{\Omega, \Omega_i}$ and $\IdId_V$,
and $\phi_{\Omega, \Omega_i}$ is linear in $\Omega$,
the vector space $H$ is generated by $\phi_{\Omega_i, \Omega_j}, i<j$,
and $\IdId_V$. Therefore, $H$ is at most 4-dimensional. 
{}From \eqref{_phi_through_projectors_Equation_} it is clear
that for any non-degenerate $\Omega_1$, $\Omega_2$,
the operator $\phi_{\Omega_1, \Omega_2}$ can be expressed
through $\phi_{\Omega_2, \Omega_1}= \phi_{\Omega_1, \Omega_2}^{-1}$
and $\IdId_V$: 
\begin{equation}\label{_reverse_or_Equation_}
  \phi_{\Omega_2, \Omega_1} = a \phi_{\Omega_2, \Omega_1} + b \IdId_V.
\end{equation}
Since 
\begin{equation}\label{_multiply_phi_Equation_}
  \phi_{\Omega_i, \Omega_j}\circ \phi_{\Omega_j, \Omega_k}=\phi_{\Omega_i,\Omega_k},
\end{equation}
the space $H$ is a subalgebra in $\End(V)$
(to multiply some of $\phi_{\Omega_i, \Omega_j}$ and
$\phi_{\Omega_{i'}, \Omega_{j'}}$, you would have to reverse
the order when necessary, using \eqref{_reverse_or_Equation_}, 
and then apply \eqref{_multiply_phi_Equation_}).

We proved that $H$ is an algebra, spanned by
$\phi_{\Omega_i, \Omega_j}$, and $\dim H \leq 4$.

\hfill

{\bf Step 2:}
We prove that a general element of $H$ can be written 
as $h= \phi_{\Omega, \Omega'}$, 
for some $\Omega, \Omega'\in \bOmega$.
Indeed, as we have shown, a general element of $H$ has form 
\begin{equation}\label{_h_genera_Equation_} 
h= a\phi_{\Omega_1, \Omega_2} + b\phi_{\Omega_1, \Omega_3} 
   + c \phi_{\Omega_2, \Omega_3} + d \IdId_V,
\end{equation}
where $\Omega_1, \Omega_2, \Omega_3$ is a basis of non-degenerate forms for $\bOmega$.
Since $\phi$ is linear in the first argument, this gives
\begin{equation}\label{_bOmega_1+cOmega_2_Equation_}
  h= a\phi_{\Omega_1, \Omega_2} + \phi_{b\Omega_1+c\Omega_2, \Omega_3} + d c \IdId_V.
\end{equation}
If the form $b\Omega_1+c\Omega_2$ is non-degenerate, we use
the reversal as indicated in \eqref{_reverse_or_Equation_}, 
obtaining
\[
\phi_{b\Omega_1+c\Omega_2, \Omega_3}=\lambda\phi_{\Omega_3,b\Omega_1+c\Omega_2}
+ \lambda'\IdId_V,
\]
write, similarly,
\[ 
  a\phi_{\Omega_1, \Omega_2}=\mu \phi_{\Omega_1, b\Omega_1+c\Omega_2}
  + \mu' \lambda'\IdId_V,
\]
then, adding the last two formulae, obtain
\[
h= (\mu +1)\phi_{\Omega_1+\Omega_3, b\Omega_1+c\Omega_2} + (\lambda'+\mu'+d)\IdId_V.
\]
Finally, $\phi_{\Omega, \Omega'}+ c\IdId_V= \phi_{\Omega, \Omega'+c\Omega}$.
This implies that a general $h\in H$ can be written as 
$\phi_{\Omega, \Omega'}$, for appropriate $\Omega,\Omega'$.

\hfill

{\bf Step 3:} We prove that $H$ is a quotient of a Clifford algebra.
From Step 2 and \eqref{_phi_matrix_Equation_}
it follows that a characteristic polynomial of any $h\in H$
has form
\begin{equation}\label{_H_charact_Equation_}
\Char_h(t)=P_h(t)^{2n},
\end{equation}
where $4n=\dim V$, and $P_h(t) = t^2 + a(h) t +b(h)$ 
is a quadratic polynomial with roots $\lambda, \lambda'$.
The map $h \arrow b(h)$ is a homogeneous polynomial 
function of degree 2 on $H$, and $a(h)$ is linear,
$a(h) = -\frac{\Tr(h)}{2n}$. From \eqref{_phi_matrix_Equation_}
the following analogue of Cayley-Hamilton theorem is apparent:
$P_h(h)=0$, giving $h^2- a(h) h + b(h)=0$.
Applying this to $h=x+y$, we obtain 
\begin{equation}\label{_HC_to_sum_Equation_}
(x+y)^2-(a(x)+a(y))(x+y)+ b(x+y)=0.
\end{equation}
Denote by $q(x,y)$ the bilinear form $q(x,y):= b(x+y)-b(x)-b(y)$.
Consider the homogeneous part of the equation \eqref{_HC_to_sum_Equation_}
of homogeneity 1 on x and y. We obtain
that the following equation holds for all $x, y \in H$:
\begin{equation}\label{_CH_homo_anticomm_Equation_}
xy + yx=(a(x) y + a(y) x- q(x,y)).
\end{equation}

Choose now $x, y \in H$ in such a way that $a(x)=a(y)=0$.
We pick $x = \phi_{\Omega_1, \Omega_2}$ and $y= \phi_{\Omega_2, \Omega_3}$,
for linearly independent, non-degenerate $\Omega_i$.
Then adjust $\Omega_1$ and $\Omega_3$ by replacing it with
$\Omega_1+ c \Omega_2$ and $\Omega_3+ c' \Omega_2$ in such a way
that $a(x)=a(y)=0$ (this is possible, as follows from
\eqref{_phi_matrix_Equation_}). 
Then \eqref{_CH_homo_anticomm_Equation_} gives
$xy+yx=-q(x,y)$. We have shown that $x, y$ satisfy 
relations for the Clifford algebra.
Also, $x$ and $y$ generate $H$ (Step 2). This gives a surjection
$\Psi:\; \Cl(W,q)\arrow H$, where $W=\langle x,y\rangle$.

\hfill

{\bf Step 4:}
We have constructed a surjective homomorphism
$\Psi:\; \Cl(W,q)\arrow H$, and now we are going
to show that it is also injective. 

Notice that the space
$H_{\Omega'}:=\bigcup_{\Omega\in \bOmega} \phi_{\Omega, \Omega'}$ 
is 3-dimensional, because the operator 
$\phi_{\Omega, \Omega'}$ uniquely determines the form $\Omega$,
if $\Omega'$ is fixed: 
\[ \phi_{\Omega, \Omega'} (\Omega')=\Omega.
\]
Then $\Psi(\phi_{\Omega,\Omega'})\neq 0$ for any
$\Omega,\Omega'$. This implies that the ideal
$\ker \Psi$ is at most 1-dimensional.

To see that $\ker \Psi=0$, 
it remains to show that $\Cl(W)$ has no 1-dimensional ideals.
If $q$ is non-degenerate on $W$, the Clifford 
algebra $\Cl(W,q)$ is isomorphic to $\Mat(2)$;
this algebra is simple, hence any non-zero homomorphism
$\Psi:\; \Cl(W,q)\arrow H$ is injective.

If $q$ is degenerate, let $x, y$ be an orthogonal basis in $W$,
$q(x,x)=1$, $q(y,y)=0$. We are going to show that $\Cl(W)$ 
contains no 1-dimensional ideals.

Let $a,b,c,d\in \C$ be any numbers.
Then $y\cdot (a +bx + cy +d xy)= ay -bxy$ and 
$(a +bx + cy +d xy)\cdot y=ay +bxy$. This implies that
any ideal containing $a +bx + cy +d xy$ with $a,b\neq 0$
is at least 2-dimensional. 

Now, $x\cdot (cy+dxy)=cxy +dy$ and $(cy+dxy)\cdot x=-cxy +dy$.
Therefore, $c=0$ and any 1-dimensional ideal is generated by 
$xy$, but then in contains $y$, giving a contradiction.

We proved Proposition \ref{_trisymple_Cliff_Proposition_}. 
\end{proof}

%%%%%%%%%%%%%%%%%%%%%%%%%%%%%%%%%%%%%%%%%%%%%%%%%%%%%%%%%%%%
\Claim\label{_Mat_2_trisy_Claim_}
Under the hypotheses of Proposition \ref{_trisymple_Cliff_Proposition_},
the weakly trisymplectic structure is trisymplectic if and only if
the Clifford algebra $H=\Cl(W)$ is isomorphic to $\Mat(2)$.
\ec

\begin{proof} Given a non-degenerate form $\Omega\in \bOmega$, let
$H_{\Omega}\subset H$ the subspace
$H_{\Omega}:=\bigcup_{\Omega'\in \bOmega} \phi_{\Omega', \Omega}$. 
The form $\Omega_1\in \bOmega$ is degenerate if and only if
$\det(\phi_{\Omega,\Omega'})=0$. However, in $H$ one has
$\det(h)= b(h)^{2n}$, where $b(h)$ is the quadratic polynomial
defined in Step 3 of the proof of Proposition \ref{_trisymple_Cliff_Proposition_}.
From Proposition \ref{_trisymple_Cliff_Proposition_} we obtain that
$H$ is isomorphic either to $\Mat(2)=\Cl(W,+,+)$ (the Clifford algebra
of a 2-dimensional space with definite scalar product) or to $\Cl(W,+,0)$
(the Clifford algebra of a 2-dimensional space with scalar product 
of rank 1). In the first case, any invariant 2-form is non-degenerate, while
in the second case, any invariant 2-form has rank 2. 

By definition, $\bOmega$ is trisymplectic if the restriction
$b\restrict {H_\bOmega}$ is non-degenerate. However, for
$\Cl(W,+,0)$ this is impossible, because $b$ has rank 2.
This implies that for a trisymplectic structure, one has
$H\cong \Mat(2)$.

To prove the converse, consider a weakly trisymplectic
structure $\bOmega\subset \Lambda^2 V$ with $H\cong \Mat(2)$.
As in Proposition \ref{_trisymple_Cliff_Proposition_}, Step 3,
find non-proportional forms $\Omega_1, \Omega_2 \in \bOmega$
such that $a(x)=a(y)=0$,
where $x=\phi_{\Omega,\Omega_1}$ and $y=\phi_{\Omega,\Omega_2}$,
and $a(h) =-\frac{\Tr(h)}{2n}$
As shown in Proposition \ref{_trisymple_Cliff_Proposition_}, Step 3,
$H\cong \Cl(W,b)$, where $W=\langle x, y\rangle$.
Since $H_0=\langle 1, x,y\rangle$, to prove that
$\bOmega$ is trisymplectic it remains to show that
the quadratic form $b\restrict{H_0}$ is non-degenerate.
Since an invariant scalar product on $\Mat(2)$
is unique, up to a constant, it follows that $b(h)$ 
is proportional to the determinant on $\Mat(2)$.
Therefore, the corresponding scalar product can
be written as $q(t, t')=\Tr(tt')$.
On $\Cl(W)$, one has $\Tr(x)=\Tr(y)=0$, hence
$1\bot_q\langle x, y\rangle$. The restriction 
$q\restrict {\langle x, y\rangle}$ is non-degenerate
by Proposition \ref{_trisymple_Cliff_Proposition_}, Step 3.
Therefore, $q$ restricted to a 3-dimensional space
$H_0=\langle 1, x,y\rangle$ is also non-degenerate.
We proved that $\bOmega$ is trisymplectic
whenever $H\cong \Mat(2)$. \end{proof}

%%%%%%%%%%%%%%%%%%%%%%%%%%%%%%%%%%%%%%%%%%%%%%%%%%%%%%%%%%%%
\Remark\label{_Lie_H_action_Remark_}
Let $(V,\bOmega)$ be a trisymplectic vector space,
and let $H\cong \Mat(2)$ be the algebra
constructed in Claim \ref{_Mat_2_trisy_Claim_}.
Then $\bOmega$ is invariant under the Lie algebra 
action induced by $H$. Moreover,
there exists a non-degenerate, $\g$-invariant 
quadratic form $Q$ on $\bOmega$, unique up to a constant, 
such that $\Omega\in \bOmega$ is degenerate if and
only if $Q(\Omega,\Omega)=0$. Indeed, the space of
degenerate forms in $\bOmega$ is a non-degenerate
quadric.
\er

In a similar way, we obtain the following useful corollary.

%%%%%%%%%%%%%%%%%%%%%%%%%%%%%%%%%%%%%%%%%%%%%%%%%%%%%%
\Claim\label{_trisymple_subspa_Claim_}
Let $(V, \bOmega)$ be a trisymplectic space, and $W\subset V$
a complex subspace. Then $\bOmega\restrict W$
is a trisymplectic space if and only if the following two
assumptions hold.
\begin{description}
\item[(i)] The space $W$ is $H$-invariant, where $H\cong \Mat(2)$
is the subalgebra of $\End(V)$ constructed in Proposition 
\ref{_trisymple_Cliff_Proposition_}.
\item[(ii)] A general 2-form $\Omega\in \bOmega$ is non-degenerate on $W$.
\end{description} 
\ec
{\bf Proof:} Let $Z\subset H$ be the set of idempotents
in $H$. Consider the standard action of $\g\cong \goth{sl}(2)$
on $V$ constructed in Proposition \ref{_trisymple_Cliff_Proposition_}. 
Clearly, $V$ is a direct sum of several 2-dimensional
irreducible representations of $\goth{sl}(2)$.

It is easy to see that for every $\Pi\in Z$
there exists a Cartan subalgebra $\goth h\subset \g$
such that $\Pi$ is a projection of $V$ onto one of
two weight components of the weight decomposition
associated with $\goth h$. If $W\subset V$ is an $H$-submodule,
then $\Pi\restrict W$ is a projection to a weight component
$W_0\subset W$ of dimension $\frac 1 2 \dim W$. 
{}From \eqref{_phi_through_projectors_Equation_} 
it is also clear that for any degenerate form $\Omega\in \bOmega$,
an annihilator of a restriction $\Omega\restrict W$ is equal to 
the weight component $W_0$, for an appropriate choice of 
Cartan subalgebra. Therefore, 
\[ \dim \left(\Ann \Omega\restrict W\right) = \frac 1 2 \dim W.
\] Similarly, \eqref{_phi_through_projectors_Equation_} 
implies that a non-\-de\-ge\-ne\-rate form is restricted to a 
non-\-de\-ge\-ne\-rate  form. We obtain that the restriction 
$\bOmega\restrict W$ to an $H$-submodule is always
a trisymplectic structure on $W$. 

To obtain the converse statement, take two non-degenerate,
non-collinear forms $\Omega_1, \Omega_2\in \bOmega$, and notice that
there exist precisely two distinct numbers $t =\lambda, \lambda'$
for which $\Omega_1 + t\Omega_2$ is degenerate
(see Proposition \ref{_trisymple_Cliff_Proposition_}, Step 1).
Let $S, S'$ be the corresponding annihilator spaces.
By construction, the algebra
$H$ is generated, as a linear space, by the projection operators
$\Pi_{S, S'}$, projecting $V$ to $S$ along $S'$. For
any $W\subset V$ such that the restriction
$\bOmega\restrict W$ is trisymplectic, one has 
$W=S\cap W \oplus S'\cap W$, hence $\Pi_{S,S'}$ 
preserves $W$. Therefore, $W$ is an $H$-submodule.
\endproof

%%%%%%%%%%%%%%%%%%%%%%%%%%%%%%%%%%%%%%%%%%%%%%%%%%%%%%%%%%%%%
\Definition\label{_non_dege_trisy_Definition_}
Let $(V, \bOmega)$ be a trisymplectic space,
and $W\subset V$ a vector subspace.
Consider the action of $H\simeq \Mat(2)$ on $V$ induced
by the trisymplectic structure. A subspace
$W\subset V$ is called {\bf non-degenerate}
if the subspace $H\cdot W\subset V$ is trisymplectic.
\ed

%%%%%%%%%%%%%%%%%%%%%%%%%%%%%%%%%%%%%%%%%%%%%%%%%%%%%%%%%%%%%
\Remark\label{_non_dege_form_Remark_}
By Claim \ref{_trisymple_subspa_Claim_}, $W$ is non-degenerate
if and only if the restriction of $\Omega$ to $H\cdot W$ is non-degenerate for some $\Omega\in \bOmega$.
\er

%%%%%%%%%%%%%%%%%%%%%%%%%%%%%%%%%%%%%%%%%%%%%%%%%%%%%%%%%%%%
\subsection[Trisymplectic structures and invariant quadratic forms on 
vector spaces with $\Mat(2)$-action]
{Trisymplectic structures and invariant \\ quadratic forms on 
vector spaces with $\Mat(2)$-action}
%%%%%%%%%%%%%%%%%%%%%%%%%%%%%%%%%%%%%%%%%%%%%%%%%%%%%%%%%%%%

Let $V$ be a complex vector space with a standard action of
the matrix algebra $\Mat(2)$, i.e. $V \cong V_0\otimes \C^2$ and $\Mat(2)$ acts only through the second factor. 
An easy way to obtain a trisymplectic structure
is to use non-degenerate, invariant quadratic forms on $V$.

Consider the natural $SL(2)$-action on $V$ induced by $\Mat(2)$,
and extend it multiplicatively to all tensor powers of $V$.
Let $g\in \Sym^2_\C(V)$ be an $SL(2)$-invariant, non-degenerate
quadratic form on $V$, and let $\{I,J,K\}$ be a quaternionic basis in
$\Mat(2)$, i.e. $\{\IdId_V,I,J,K\}$ is a basis for $\Mat(2)$ and $I^2=J^2=K^2=IJK=-1$. Then 
\[
g(x, Iy) = g(Ix, I^2y) = - g(Ix, y)
% g(x, Iy) = g(y,Ix) = g(Iy, I^2x) = - , ...must be wrong, compare 2nd with last term...
\]
hence the form $\Omega_I(\cdot, \cdot):= g(\cdot, I\cdot)$
is a symplectic form, obviously non-degenerate; similarly, the forms
$\Omega_J(\cdot, \cdot):= g(\cdot, J\cdot)$ and $\Omega_K(\cdot, \cdot):= g(\cdot, K\cdot)$
have the same properties. It turns out that this construction gives a trisymplectic
structure, and all trisymplectic structures can be obtained in this way.

%%%%%%%%%%%%%%%%%%%%%%%%%%%%%%%%%%%%%%%%%%%%%%%%%%%%%%%%%%%%
\Theorem\label{_quadra_from_trisy_Theorem_}
Let $V$ be a vector space equipped with a standard action of
the matrix algebra $\Mat(2)\stackrel\rho\arrow \End(V)$, and
$\{I,J,K\}$ a quaternionic basis in
$\Mat(2)$. Consider the corresponding action of $SL(2)$
on the tensor powers of $V$. Then
\begin{description}
\item[(i)] Given a non-degenerate, $SL(2)$-invariant quadratic
form $g\in \Sym^2(V)$, consider the space $\bOmega\subset\Lambda^2 V$
generated by the symplectic forms $\Omega_I, \Omega_J, \Omega_K$
defined as above,
\begin{equation}\label{_quadra_indu_trise_Equation_}
\Omega_I(\cdot, \cdot):= g(\cdot, I\cdot),\ \ 
\Omega_J(\cdot, \cdot):= g(\cdot, J\cdot),\ \ 
\Omega_K(\cdot, \cdot):= g(\cdot, K\cdot).
\end{equation}
Then $\bOmega$ is a trisymplectic structure on $V$, with
the operators $\Omega_K^{-1}\circ \Omega_J$ and $\Omega_K^{-1} \circ \Omega_I$,  generating the algebra
$H \cong \Mat(2) := \im(\rho)\subset \End(V)$ as in 
Proposition \ref{_trisymple_Cliff_Proposition_}.
\item[(ii)] Conversely, for each trisymplectic structure
$\bOmega$ inducing the action of $H\cong \Mat(2)$ on $V$ given by $\rho$, there exists a unique
(up to a constant) $SL(2)$-invariant non-degenerate 
quadratic form $g$ inducing $\bOmega$ as in 
\eqref{_quadra_indu_trise_Equation_}.
\end{description}
\et

{\bf Proof:} First, consider the 3-dimensional subspace of $\Lambda^2V$ generated by 
$\Omega_I, \Omega_J, \Omega_K$. Regard $\Omega_I$ as an operator from $V$ to $V^*$, $x\mapsto \Omega_I(x,\cdot)$, and similarly for
$\Omega_J$ and $\Omega_K$; let $h:= \Omega_K^{-1} \circ \Omega_J \in \End(V)$. Then
$$ h(x) = \Omega_K^{-1}(-g(Jx,\cdot)) = -KJx = Ix , $$
hence $h=I$. Similarly, one concludes that $\Omega_K^{-1} \circ \Omega_I = J$, hence $\Omega_K^{-1} \circ \Omega_J$ and 
$\Omega_K^{-1} \circ \Omega_I$ generate $H$ as an algebra.

To complete the proof of the first claim of the Theorem,
it remains for us to show that $\bOmega$ is a trisymplectic structure.
For this it would suffice to show that any non-zero, degenerate
form $\Omega\in \bOmega$ has rank $\frac 1 2 \dim V$. 
Consider $V$ as a tensor product $V= V_0 \otimes \C^2$,
with $\Mat(2)$ acting on the second factor. 
Choose a basis $\{x,y\}$ in $\C^2$, so that
$V= V_0 \otimes x\oplus V_0 \otimes y$.
{}From $SL(2)$-invariance it is clear that
$g(v_0\otimes \zeta)= g(v_0\otimes \xi)$
for any non-zero $\zeta,\xi \in \C^2$.
Therefore, $V_0 \otimes x\subset V$
and $V_0 \otimes y\subset V$ are isotropic subspaces,
dual to each other. Denote by $\Omega_{V_0}$ the
corresponding bilinear form on $V_0$:
\[ \Omega_{V_0}(v,v'):=g(v\otimes x, v'\otimes y).\]
Since the group $SL(2)$ acts transitively on
the set of all $\zeta,\xi \in \C^2$ satisfying
$\zeta\wedge \xi = x\wedge y$, we obtain
\[
\Omega_{V_0}(v,v')=g(v\otimes x, v'\otimes y)= -g(v\otimes y, v'\otimes x)=
-\Omega_{V_0}(v', v).
\]
Therefore, $\Omega_{V_0}$ is skew-symmetric. Conversely, 
$g$ can be expressed through $\Omega_{V_0}$, as follows.
Given $x',y'\in\C^2$ such that $x'\wedge y'\ne0$, and $v\otimes x_1, w\otimes y_1\in V$, we find 
$h\in SL(2)$ such that $h(x')=\lambda x$ and $h(y')=\lambda y$ with $\lambda= \frac{x_1\wedge y_1}{x\wedge y}$.
Since $g$ is $SL(2)$-invariant, one has 
\[ g(v\otimes x', w\otimes y')=\lambda^2 g(v\otimes x, w\otimes y).
\]
Therefore, for appropriate symplectic form $\Omega_{\C^2}$
on $\C^2$, one would have
\begin{equation}\label{_g_via_Omega_v_0_Equation_}
g(v\otimes x', w\otimes y')= \Omega_{V_0}(v,w) \cdot \Omega_{\C^2}(x', y').
\end{equation}

This gives us a description of the group $\St(H,g)\subset \End(V)$ which fixes 
the algebra $H\subset \End(V)$ and $g$. Indeed, from
\eqref{_g_via_Omega_v_0_Equation_},
we obtain that $\St(H,g)\cong \Sp(V_0, \Omega_{V_0})$
acting on $V=V_0\otimes \C^2$ in a standard way, i.e. trivially
on the second factor.

%Return now to the proof of
%Theorem \ref{_quadra_from_trisy_Theorem_} (i). From $g$ and $H$ 
%we have constructed a 3-dimensional space of 2-forms
%$\bOmega\subset \Lambda^2 V$, and to prove
%Theorem \ref{_quadra_from_trisy_Theorem_} (i) we need only to check that
%all forms which lie in $\bOmega$ have appropriate rank.

Since all elements of $\bOmega$ are by construction fixed 
by $\St(H,g)\cong\Sp(V_0, \Omega_{V_0})$, for any $\Omega\in \bOmega$,
the annihilator of $\Omega$ is $\Sp(V_0, \Omega_{V_0})$-invariant.
However, $V\cong V_0\oplus V_0$ is isomorphic to a sum of two copies of
the fundamental representation of $\Sp(V_0, \Omega_{V_0})$,
hence any $\Sp(V_0, \Omega_{V_0})$-invariant space has
dimension $0, \frac 1 2 \dim V$, or $\dim V$. We finished the
proof of Theorem \ref{_quadra_from_trisy_Theorem_} (i).

\hfill

The proof of the second part of the Theorem is divided into several steps.

{\bf Step 1.}
%Starting from a trisymplectic structure on a vector
%space, we construct a symmetric 2-form $g_I$.
Let $I\in\Mat(2)$ be such that $I^2=-\IdId_V$.
Consider  the action 
$\rho_I:\; U(1) \arrow \End(V)$ generated by 
$t \arrow \cos t \IdId_V + \sin t \rho(I).$
As shown in Remark \ref{_Lie_H_action_Remark_}, 
$\bOmega$ is an $SL(2)$-subrepresentation of $\Lambda^2V$.
This representation is by construction irreducible.
Since it is 3-dimensional, it is isomorphic to the
adjoint representation of $SL(2)$; let $\phi:\; \goth{sl}(2)\arrow \bOmega$ be an isomorphism.
Therefore, there exists a 2-form $\Omega_I\in \bOmega$
fixed by the action of $\rho_I$, necessarily unique up
to a constant multiplier.
Write $g_I(x,y):= -\Omega_I(x, Iy)$. Then
\[
g_I(y,x) = \Omega_I(y, Ix)=-\Omega_I(Ix,y) = -\Omega_I(I^2x,Iy) = 
\Omega(x,Iy)=g_I(x,y),
\]
hence $g_I$ is symmetric, i.e. $g_I\in \Sym^2_\C(V)$.

\hfill

{\bf Step 2.} Now let $\{I,J,K\}$ be the quaternionic basis for $\Mat(2)$.
We prove that the symmetric tensor $g_I$ constructed in Step 1 is fixed by the subgroup
$\{ \pm 1, \pm I, \pm J, \pm K\}\subset SL(2)\subset \Mat(2)$,
for an appropriate choice of $\Omega_I\in \bOmega$.

Using the $SL(2)$-invariant isomorphism $\phi:\; \goth{sl}(2)\arrow \bOmega$ constructed in Step 1,
and the identification of $\goth{sl}(2)$ with the subspace of $\Mat(2)$ generated by 
$I$, $J$ and $K$, we fix a choice of $\Omega_I$
by requiring that $\phi(I)=\Omega_I$. 
Then, $J$ and $K$, considered as elements
of $SL(2)$, act on $\Omega_I$ by $-1$:
\[
\Omega_I(J\cdot, J\cdot) = - \Omega_I(\cdot,\cdot),\ \ 
\Omega_I(K\cdot, K\cdot) = - \Omega_I(\cdot,\cdot).
\]
This gives 
\[
g_I(J\cdot, J\cdot) = \Omega_I(J\cdot, IJ\cdot)=-
\Omega_I(J\cdot, JI\cdot)=\Omega_I(\cdot,I\cdot) =g(\cdot,\cdot).
\]
We have shown that $J$, considered as an element
of $SL(2)$, fixes $g_I$. The same argument applied to $K$
implies that $K$ also fixes $g_I$. We have shown that
$g_I$ is fixed by the Klein subgroup
${\goth K}:=\{ \pm 1, \pm I, \pm J, \pm K\}\subset SL(2)$.

\hfill

{\bf Step 3.} We prove that $g_I$ is $SL(2)$-invariant.

Consider $\Sym^2 V$ as a representation of $SL(2)$. Since
$V$ is a direct sum of weight 1 representations, 
Clebsch-Gordon theorem implies that $\Sym^2 V$
is a sum of several weight 2 and trivial representations.
However, no element on a weight 2 representation
can be ${\goth K}$-invariant. Indeed, a weight 2 representation
$W_2$ is isomorphic to an adjoint representation, that is,
a complex vector space generated by the imaginary
quaternions: $W_2:=\langle I, J, K\rangle\subset \Mat(2)$ 
Clearly, no non-zero linear combination of $I,J,K$. 
can be ${\goth K}$-invariant. Since $g_I$ is ${\goth K}$-invariant,
this implies that $g_I$ lies in the $SL(2)$-invariant part of
$\Sym^2_\C V$.

\hfill

{\bf Step 4.} We prove that $g_I$ is proportional to $g_{I'}$, 
for any choice of quaternionic triple $I',J',K'\in \Mat(2)$.
The ambiguity here is due to the
ambiguity of a choice of $\Omega_I$ in a centralizer of
$\rho_I$. The form $\Omega_I$ is defined up to
a constant multiplier, because this centralizer is $1$-dimensional.

The group $SL(2)$ acts transitively on the set of quaternionic
triples. Consider $h\in SL(2)$ which maps $I,J,K$ to 
$I',J',K'\in \Mat(2)$. Then $h(g_I)$ is proportional
to $g_{I'}$.

{\bf Step 5:} To finish the proof of 
Theorem \ref{_quadra_from_trisy_Theorem_} (ii), it 
remains to show that the $SL(2)$-invariant quadratic form $g$
defining the trisymplectic structure $\bOmega$ is unique, up
to a constant.  Indeed, let $g$ be such a form; then
$g=\Omega(\cdot, I\cdot)$, for some $\Omega\in \bOmega$ and 
$I\in \Mat(2)$, satisfying $I^2=-1$. 
Since $g$ is $SL(2)$-invariant, the form
$\Omega$ is $\rho_I$-invariant, hence 
$g$ is proportional to the form $g_I$
constructed above. We proved Theorem \ref{_quadra_from_trisy_Theorem_}.
\endproof

%%%%%%%%%%%%%%%%%%%%%%%%%%%%%%%%%%%%%%%%%%%%%%%%%%%%%%%%%%%%

\section{$SL(2)$-webs and trisymplectic structures}

%%%%%%%%%%%%%%%%%%%%%%%%%%%%%%%%%%%%%%%%%%%%%%%%%%%%%%%%%%%%

In this section we introduce the notion of trisymplectic structures on manifolds,
study its reduction to quotients, and explain how they are related to holomorphic $SL(2)$-webs. 

The trisymplectic structures and trisymplectic reduction were previously considered in a
context of framed instanton bundles by Hauzer and Langer
(\cite[Sections 7.1 and 7.2]{_Hauzer_Langer_}). However, their approach
is significantly different from ours, because
they do not consider the associated $SL(2)$-web structures.

%%%%%%%%%%%%%%%%%%%%%%%%%%%%%%%%%%%%%%%%%%%%%%%%%%%%%%%%%%%%%
\subsection{Trisymplectic structures on manifolds}
%%%%%%%%%%%%%%%%%%%%%%%%%%%%%%%%%%%%%%%%%%%%%%%%%%%%%%%%%%%%%

%%%%%%%%%%%%%%%%%%%%%%%%%%%%%%%%%%%%%%%%%%%%%%%
\Definition\label{_trisymple_mani_Definition_}
{\bf A (weakly) trisymplectic structure} on an even dimensional complex
manifold $M$ is a 3-dimensional space $\bOmega\subset \Omega^2M$
of closed holomorphic $2$-forms 
such that at any $x\in M$, the evaluation
$\bOmega(x)$ gives a (weakly) trisymplectic
structure on the tangent space $T_x M$.
A complex manifold equipped with
a (weakly) trisymplectic structure is called a
{\bf (weakly) trisymplectic manifold}.
\ed

%%%%%%%%%%%%%%%%%%%%%%%%%%%%%%%%%%%%%%%%%%%%%%%

Clearly, trisymplectic manifolds must have even complex dimension. Notice also
that Theorem \ref{_quadra_from_trisy_Theorem_} implies the equivalence
between the Definition above and Definition \ref{_trisy_intro_Definition_}.

A similar notion is called a \emph{hypersymplectic structure}
by Hauzer and Langer in \cite[Definition 7.1]{_Hauzer_Langer_}.
A complex manifold $(X,g)$ equipped with a non-degenerate holomorphic
symmetric form 
is called {\bf hypersymplectic} in \cite{_Hauzer_Langer_} if there are
three complex structures $I$, $J$ and $K$ satisfying quaternionic relations and
$g(Iv, Iw) = g(Jv, Jw) = g(Kv,Kw) = g(v,w)$. Clearly, one can then define three
nondegerate symplectic forms $\omega_1(v,w) = g(Iv,w)$, $\omega_2(v,w) = g(Jv,w)$
and $\omega_3(v,w) = g(Kv,w)$ which generate a $3$-dimensional subspace of holomorphic $2$-forms $\bOmega\subset \Omega^{2,0}X$. If every nonzero, degenerate linear combination of $\omega_1$, $\omega_2$ and $\omega_3$ has rank $\dim X/2$, then $(X,\bOmega)$ is a weakly trisymplectic manifold. Furthermore, if in addition 
every degenerate form in $\Omega$ belong to a non-degenerate quadric
hypersurface in $\Omega^{2,0}X$, then $(X,\bOmega)$ is a trisymplectic manifold.
 
We prefer, however, to use the term ``trisymplectic"'
to avoid confusion with the hypersymplectic structures
known in differential geometry (see \cite{_Andrada_Dotti_,_Dancer_Swann:hypersy_}), 
where a hypersymplectic structure is a $3$-dimensional space $W$ 
of differential $2$-forms on a real manifold which satisfy the 
same rank assumptions as in Definition \ref{_trisymple_mani_Definition_}, and, in addition,
contain a non-trivial degenerate $2$-form 
(for complex-linear 2-forms, this last assumption
is automatic).

%%%%%%%%%%%%%%%%%%%%%%%%%%%%%%%%%%%%%%%%%%%%%%%
\Definition
Let $\eta$ be a $(p,0)$-form on a complex manifold $M$.
The set 
$$ \Null_\eta =\{ v\in T^{1,0}M ~|~ \eta\cntrct v=0 \}, $$
where $\cntrct$ denotes the contraction, is called {\bf the 
null-space}, or {\bf an annihilator}, of $\eta$.
\ed

%%%%%%%%%%%%%%%%%%%%%%%%%%%%%%%%%%%%%%%%%%%%%%%%%%%%
\Lemma\label{_null_space_invo_Lemma_}
Let $\eta$ be a closed $(p,0)$-form for which $\Null_\eta$
is a sub-bundle in $T^{1,0}(M)$. Then $\Null_\eta$
is holomorphic and involutive, that is, satisfies
\[ [\Null_\eta, \Null_\eta]\subset \Null_\eta.\]
\el

{\bf Proof:} The form $\eta$ is closed and hence holomorphic, therefore
$\Null_\eta$ is a holomorphic bundle. To prove that
$\Null_\eta$ is involutive, we use the Cartan's formula,
expressing de Rham differential in terms of commutators
and Lie derivatives. Let $X\in T^{1,0}(M)$, $Y, Z\in \Null_\eta$.
Then Cartan's formula gives $0 = d\eta(X,Y,Z) = \eta(X, [Y,Z])$.
This implies that $[Y,Z]$ lies in $\Null_\eta$.
\endproof

\hfill

Lemma \ref{_null_space_invo_Lemma_} can be used to construct
holomorphic $SL(2)$-webs on manifolds, as follows.

%%%%%%%%%%%%%%%%%%%%%%%%%%%%%%%%%%%%%%%%%%%%%%%%%%%%
\Theorem \label{_trisy_web_Theorem_}
Let $M$ be an even dimensional complex manifold, and let $\bOmega\subset \Omega^{2,0}(M)$ be a trisymplectic structure on $M$.
Then there is a holomorphic $SL(2)$-web  $(M, S_t)$, $t\in \C \p1$ 
on $M$ such that each sub-bundle $S_t$ is a null-space of
a certain $\Omega_t\in \bOmega$.
\et

{\bf Proof:}
Theorem \ref{_trisy_web_Theorem_}
follows immediately from Proposition 
\ref{_trisymple_Cliff_Proposition_}, Claim \ref{_Mat_2_trisy_Claim_},
and Lemma \ref{_null_space_invo_Lemma_}. Indeed, at any point
$x\in M$, the 3-dimensional space $\bOmega(x) \in \Lambda^{2,0}(T_xM)$
satisfies assumptions of Claim \ref{_Mat_2_trisy_Claim_},
hence it induces an action of the matrix algebra $H \cong \Mat(2)$
on $T_xM$. Denote by $Z\subset \bOmega$ the set of degenerate forms.
{}From Remark \ref{_Lie_H_action_Remark_}
we obtain that the projectivization ${\mathbb P}Z\subset {\mathbb P}\bOmega$
is a non-singular quadric, isomorphic to $\C \p1$.
For each $t\in \Z$, the corresponding zero-space
$S_t \subset TM$ is a sub-bundle of dimension $\frac 1 2 \dim M$,
and for distinct $t$, the bundles $S_t$ are obviously transversal.
Also, Lemma \ref{_null_space_invo_Lemma_} implies that
the bundles $S_t$ are involutive. Finally, the projection
operators associated to $S_t, S_t'$ generate a subalgebra
isomorphic to $\Mat(2)$, as follows from 
Claim \ref{_Mat_2_trisy_Claim_}.
We have shown that $S_t, t\in {\mathbb P}Z\cong \C \p1$
is indeed a holomorphic $SL(2)$-web.
\endproof

In particular, every trisymplectic 
manifold has an induced holomorphic $SL(2)$-web. 

%%%%%%%%%%%%%%%%%%%%%%%%%%%%%%%%%%%%%%%%%%%%%%%%%%%%
\Definition
Let $(M, S_t)$, $t\in \C \p1$, be a complex manifold equipped
with a holomorphic $SL(2)$-web. Assume that there 
is a trisymplectic structure $\bOmega\subset \Omega^{2,0}(M)$
such that for each $t\in \C \p1$ there exists $\Omega_t\in \bOmega$
satisfying $S_t= \Null_{\Omega_t}$.
Then $\bOmega$ is called {\bf a trisymplectic structure
generating the $SL(2)$-web $S_t, t\in \C \p1$}. 
\ed

%%%%%%%%%%%%%%%%%%%%%%%%%%%%%%%%%%%%%%%%%%%%%%%%%%%%%%%%%%%%%
\subsection[Chern connection on $SL(2)$-webs 
and trisymplectic structures]{Chern connection on $SL(2)$-webs \\
and trisymplectic structures}
%%%%%%%%%%%%%%%%%%%%%%%%%%%%%%%%%%%%%%%%%%%%%%%%%%%%%%%%%%%%%

The following theorem is proven in the same way as one
proves that the K\"ahler forms on a hyperk\"ahler manifold
are preserved by the Obata connection. Indeed,
a trisymplectic structure is a complexification of 
a hyperk\"ahler structure, and the Chern connection
corresponds to a complexification of the Obata
connection on a hyperk\"ahler manifold.

%%%%%%%%%%%%%%%%%%%%%%%%%%%%%%%%%%%%%%%%%%%%%%%%%%%%
\Theorem\label{_trisy_Chern_Theorem_}
Let $\bOmega$ be a trisymplectic structure 
generating a holomorphic $SL(2)$-web on a complex
manifold $M$. Denote by $\nabla$ the corresponding
Chern connection. Then $\nabla\Omega=0$, for
each $\Omega\in \bOmega$.
\et

{\bf Proof:} Let $(M, \bOmega)$ be a trisymplectic manifold,
and \[ \rho:\; \goth{sl}(2) \arrow \End(\Lambda^*M)\]
the corresponding multiplicative action of $\goth{sl}(2)$ associated
to the Lie algebra $\g\cong \goth{sl}(2)\subset \End(TM)$,  
$\g=[H,H]$, where the algebra $H=\Mat(2)$ is 
constructed in Claim \ref{_Mat_2_trisy_Claim_}.
By Remark \ref{_Lie_H_action_Remark_},
$\bOmega$ is an irreducible $\goth{sl}(2)$-module.
Choose a Cartan subalgebra in $\goth{sl}(2)$, and let
$\Omega^i(M)= \bigoplus_{p+q=i} \Omega^{p,q}(M)$
be the multiplicative weight decomposition associated with this Cartan subalgebra, with
$\Omega^i(M) := \Lambda^{i,0}(M)$. We write the
corresponding weight decomposition of $\bOmega$ as 
\[ \bOmega = \bOmega^{2,0} \oplus \bOmega^{1,1}\oplus \bOmega^{0,2}.
\]
Clearly 
\begin{equation}\label{_weight_multi_forms_Equation_}
\Omega^i(M)= \bigoplus_{p+q=i}\Omega^{p,0}(M)\otimes \Omega^{0,q}(M),
\end{equation}
since $\Omega^{p,0}(M)\otimes \Omega^{0,q}(M)=\Omega^{p,q}(M)$.

Consider the Chern connection as an operator
\[ \Omega^i(M)\stackrel \nabla\arrow \Omega^i(M) \otimes\Omega^1(M)
\]
(this makes sense, because $\nabla$ is a holomorphic connection),
and let 
\[ \Omega^{p,q}(M)\stackrel {\nabla^{1,0}}\arrow \Omega^{p,q}(M) 
   \otimes\Omega^{1,0}(M), \ \ 
\Omega^{p,q}(M)\stackrel {\nabla^{0,1}}\arrow \Omega^{p,q}(M) 
   \otimes\Omega^{0,1}(M)
\]
be its weight components. Since $\nabla$ is torsion-free, one has
\begin{equation}\label{_Alt_Chern_6_Equation_}
\6\eta= \Alt(\nabla\eta),
\end{equation}
where $\6$ is the holomorphic de Rham differential,
and 
\[ \Alt:\; \Omega^i(M) \otimes\Omega^1(M)\arrow
\Omega^{i+1}(M)
\]
the exterior multiplication. Denote by $\Omega_{0,2}$,
$\Omega_{2,0}$ generators of the 1-dimensional spaces
$\bOmega^{2,0}, \bOmega^{0,2}\subset \bOmega$. 
Since $\6\Omega_{2,0}=0$, and the multiplication map
$\Omega^{0,1}(M)\otimes \Omega^{2,0}(M) \arrow \Omega^3 (M)$
is injective by \eqref{_weight_multi_forms_Equation_},
\eqref{_Alt_Chern_6_Equation_} implies that $\nabla^{0,1}(\Omega_{2,0})=0$.
Similarly, $\nabla^{1,0}(\Omega_{0,1})=0$. However, since
$\bOmega$ is irreducible as a representation of $\goth{sl}(2)$,
there exist an expression of form $\Omega_{2,0}= g(\Omega_{0,2})$,
where $g\in U_\g$ is a polynomial in $\g$.
Since the Chern connection $\nabla$ commutes with $g$, 
this implies that
\[
0 = g(\nabla^{1,0}\Omega_{0,2})= \nabla^{1,0}(g \Omega_{0,2})= 
\nabla^{1,0}\Omega_{2,0}.
\]
We have proved that both weight components of $\nabla\Omega_{2,0}$
vanish, thus $\nabla\Omega_{2,0}=0$. Acting on $\Omega_{2,0}$
by $\goth{sl}(2)$ again, we obtain that $\nabla\Omega=0$
for all $\Omega\in \bOmega$.
\endproof

%%%%%%%%%%%%%%%%%%%%%%%%%%%%%%%%%%%%%%%%%%%%%%%%%%%%%%%%%%%%%
\subsection{Trisymplectic reduction}
\label{_trisymple_red_Subsection_}
%%%%%%%%%%%%%%%%%%%%%%%%%%%%%%%%%%%%%%%%%%%%%%%%%%%%%%%%%%%%%

%%%%%%%%%%%%%%%%%%%%%%%%%%%%%%%%%%%%%%%%%%%%%%%%%%%%%%%%%%%%%
\Definition\label{_trisymple_moment_Definition_}
Let $G$ be a compact Lie group acting on a complex
manifold equipped with a trisymplectic structure $\bOmega$
generating an $SL(2)$-web. Assume that $G$ preserves
$\bOmega$. {\bf A trisymplectic moment map}
$\bMu_\C:\; M \arrow \g^* \otimes_\R \bOmega^*$
takes vectors $\Omega\in \bOmega, g\in \g=\Lie(G)$
and maps them to a holomorphic function
$f\in \calo_M$, such that $df=\Omega\cntrct g$,
where $\Omega\cntrct g$ denotes the contraction
of $\Omega$ and the vector field $g$. A moment
map is called {\bf equivariant} if it is equivariant
with respect to the coadjoint action of $G$ on $\g^*$
Further on, we shall always assume that all moment
maps we consider are equivariant. 
\ed

\hfill

%%%%%%%%%%%%%%%%%%%%%%%%%%%%%%%%%%%%%%%%%%%%%%%%%%%%%%%%%%%%%
%\Remark
Since $d\Omega=0$, and $\Lie_g\Omega=0$, Cartan's formula gives
$0=\Lie_g(\Omega) = d(\Omega\cntrct g)$, hence the contraction
$\Omega\cntrct g$ is closed. Therefore, existence of a moment
map is equivalent to exactness of this closed 1-form for each
$\Omega\in \bOmega, g\in \g=\Lie(G)$. Therefore,
the existence of a moment map is assured whenever $M$ is simply
connected. The existence of an {\em equivariant} moment map is less
immediate, and depends on certain cohomological properties
of $G$ (see e.g. \cite{_HKLR_}). 
%\er

%%%%%%%%%%%%%%%%%%%%%%%%%%%%%%%%%%%%%%%%%%%%%%%%%%%%%%%%%%%%%
\Definition 
Let $(M, \bOmega)$ be a trisymplectic manifold.
Assume that $M$ is equipped with an action
of a compact Lie group $G$ preserving $\bOmega$, and an equivariant 
trisymplectic moment map \[ \bMu_\C:\; M \arrow \g^* \otimes_\R \bOmega^*.\]
Consider a $G$-invariant vector $c\in \g^* \otimes_\R \bOmega^*$
(usually, one sets $c=0$), and let $\bMu_\C^{-1}(c)$
be the corresponding {\bf level set} of the moment map.
Consider the action of the corresponding complex Lie
group $G_\C$ on $\bMu_\C^{-1}(c)$, obtained as a
complexification. Assume that $\bMu_\C^{-1}(c)$ is smooth,
and that the action of $G_\C$ on $\bMu_\C^{-1}(c)$ is free and proper, 
so that the quotient $\bMu_\C^{-1}(c)/G_\C$ is Hausdorff 
and smooth.\footnote{In this case $\bMu_\C^{-1}(c)/G_\C$
is a complex manifold: see e.g. \cite[Example 2.1.12]{_Huybrechts:complex_}.}
Then the quotient $\bMu_\C^{-1}(c)/G_\C$ is called 
{\bf the trisymplectic quotient} of $(M,\bOmega)$,
denoted by $M\4 G$.
\ed

%%%%%%%%%%%%%%%%%%%%%%%%%%%%%%%%%%%%%%%%%%%%%%%%%%%%%%%%%%%%%
%\Remark
As we shall see, the trisymplectic quotient is related
to the usual hyperk\"ahler quotient in the same way
as the hyperk\"ahler quotient (denoted by $\3$) is related to the
symplectic quotient, denoted by $\2$. In heuristic terms,
the hyperk\"ahler quotient can be considered as a 
``complexification'' of a symplectic quotient;
similarly, the trisymplectic quotient is a 
``complexification'' of a hyperk\"ahler quotient.
%\er

%%%%%%%%%%%%%%%%%%%%%%%%%%%%%%%%%%%%%%%%%%%%%%%%%%%%%%%%%%%%%
%\Remark
The non-degeneracy condition of Theorem
\ref{_trisy_quotient_trisy_Theorem_}
below is necessary for the trisymplectic reduction process, in the
same way as one would need some non-degeneracy if one tries
to perform the symplectic reduction on a pseudo-K\"ahler manifold.
On a K\"ahler (or a hyperk\"ahler) manifold it is automatic
because the metric is positive definite, but otherwise 
it is easy to obtain counterexamples (even in the simplest
cases, such as $S^1$-action on $\C^2$ with an appropriate
pseudo-K\"ahler metric).
%\er

%%%%%%%%%%%%%%%%%%%%%%%%%%%%%%%%%%%%%%%%%%%%%%%%%%%%%%%%%%%%%
\Theorem\label{_trisy_quotient_trisy_Theorem_}
Let $(M, \bOmega)$ be a trisymplectic manifold.
Assume that $M$ is equipped with an action
of a compact Lie group $G$ preserving $\bOmega$ and a trisymplectic
moment map $\bMu_\C:\; M \arrow \g^* \otimes_\R \bOmega^*$.
Assume, moreover, that the image of $\g=\Lie(G)$ in
$TM$ is non-degenerate at any point (in the sense of
Definition \ref{_non_dege_trisy_Definition_}).
Suppose that the quotient $\bMu_\C^{-1}(c)/G_\C$ is Hausdorff.
Then the trisymplectic quotient $M\4 G$ defined
as $M\4 G:=\bMu_\C^{-1}(0)/G_\C$ 
is naturally equipped with a trisymplectic structure.
\et

For a real version of this theorem, please see 
\cite{_Dancer_Swann:hypersy_}.

The proof of Theorem \ref{_trisy_quotient_trisy_Theorem_}
takes the rest of this section. First, we shall use the
following definition and observation.

%%%%%%%%%%%%%%%%%%%%%%%%%%%%%%%%%%%%%%%%%%%%%%%%%%%%%%%%%%%%
\Definition
Let $B\subset TM$ be an involutive sub-bundle in a tangent
bundle to a smooth manifold $M$. A form $\eta\in \Omega^iM$
is called {\bf basic with respect to $B$} if for
any $X\in B$, one has $\eta \cntrct X=0$ and $\Lie_X\eta=0$.
\ed

The following claim is clear.

%%%%%%%%%%%%%%%%%%%%%%%%%%%%%%%%%%%%%%%%%%%%%%%%%%%%%%%%%%%%
\Claim\label{_basic_pullback_Claim_}
Let $B\subset TM$ be an involutive sub-bundle in a tangent
bundle to a smooth manifold $M$. Consider the projection
$M\stackrel\pi\arrow M'$ onto its leaf space, which
is assumed to be Hausdorff. Let $\eta\in \Omega^iM$
be a basic form on $M$. Then $\eta=\pi^*\eta'$, for
an appropriate form $\eta'$ on $M'$. \endproof
\ec

Return to the proof of Theorem \ref{_trisy_quotient_trisy_Theorem_}.
Let $I,J,K$ be a quaternionic basis in $\Mat(2)$,
$\Omega_I\in \bOmega$ a $\rho_I$-invariant form chosen
as in Theorem \ref{_quadra_from_trisy_Theorem_} (ii),
and $g:= \Omega_I(\cdot, I\cdot)$ the corresponding
non-degenerate, complex linear symmetric form on $M$. 
By its construction, $g$ is holomorphic, and by
Theorem \ref{_quadra_from_trisy_Theorem_} (ii), $SL(2)$-invariant.
Let $N\subset M$ be a level set of the moment map
$N:=\bMu_\C^{-1}(c)$. Choose a point $m\in N$, and let
$\g_m\subset T_mM$ be the image of $\g=\Lie G$
in $T_mM$. Then, for each $v\in \g_m$, one has
\begin{equation}\label{_d_mu_I_Equation_}
d\mu_I(v, \cdot)= \Omega_I(v, \cdot),
\end{equation}
where $\mu_I:M\to\g^*$ is the holomorphic moment map associated with the
symplectic form $\Omega_I$.

On the other hand, $\Omega_I(v, \cdot)= -g(Iv, \cdot)$. 
Therefore, $T_m N\subset T_mM$ is an orthogonal 
complement (with respect to $g$) to the space
$\langle I\g_m, J\g_m, K\g_m\rangle$ generated by
$I(\g_m), J(\g_m), K(\g_m)$:
\begin{equation}\label{_TN_as_comple_Equation_}
T_mN =\langle I\g_m, J\g_m, K\g_m\rangle^\bot_g.
\end{equation}
By \eqref{_d_mu_I_Equation_}, for any $v\in \g_m$,
and $w\in T_mN$, one has $\Omega_I(v, w)=0$.
Also, $G$ preserves all forms from $\bOmega$,
hence $\Lie_v\Omega_i=0$. Therefore, $\Omega_I$
is basic with respect to the distribution
$V \subset TN$ generated by the image of 
Lie algebra $\g\arrow TN$.

Consider the quotient map $N\stackrel \pi \arrow N/G_\C=M'$.
To prove that $M'$ is a trisymplectic manifold, 
we use Claim \ref{_basic_pullback_Claim_},
obtaining a 3-dimensional space of holomorphic 2-forms
$\bOmega'\subset \Lambda^{2,0}(M')$, with
$\bOmega\restrict N= \pi^*\bOmega'$.
To check that $\bOmega'$ is a trisymplectic structure,
it remains only to establish the rank conditions.

Let $W\subset T_mN$ be a subspace complementary
to $\g_m \subset T_mN$. Clearly, for any
$\Omega \in \bOmega$, the rank of the
corresponding form $\Omega'\in \bOmega'$ at
the point $m'=\pi(m)$ is equal to 
the rank of $\Omega\restrict W$.

Let $W_1\subset T_mM$ be a subspace obtained as
$H\cdot \g_m$, where $H\cong \Mat(2)\subset \End(T_mM)$ is
the standard action of the matrix algebra defined as in
Subsection \ref{_Mat_2_Subsection_}.
By the non-degeneracy assumption 
of Theorem \ref{_trisy_quotient_trisy_Theorem_},
the restriction $g\restrict{W_1}$ is non-degenerate,
hence the orthogonal complement
$W_1^\bot$ satisfies $T_mM = W_1 \oplus W_1^\bot$.
{}From \eqref{_TN_as_comple_Equation_} we obtain
$W_1^\bot \subset T_mN$, with $W_1^\bot \oplus \g_m= T_mN$.
Therefore, $W:=W_1^\bot$ is complementary to 
$\g_m$ in $T_mN$. The space $(W, \bOmega\restrict W)$
is trisymplectic, as follows from Claim 
\ref{_trisymple_subspa_Claim_}.
Therefore, the forms $\bOmega'\subset \Lambda^{2,0}(M')$
define a trisymplectic structure on $M'$.
We have proved Theorem \ref{_trisy_quotient_trisy_Theorem_}.
\endproof

%%%%%%%%%%%%%%%%%%%%%%%%%%%%%%%%%%%%%%%%%%%%%%%%%%%%%%%%%%%%

\section{Trihyperk\"ahler reduction}

%%%%%%%%%%%%%%%%%%%%%%%%%%%%%%%%%%%%%%%%%%%%%%%%%%%%%%%%%%%%

%%%%%%%%%%%%%%%%%%%%%%%%%%%%%%%%%%%%%%%%%%%%%%%%%%%%%%%%%%%%
\subsection{Hyperk\"ahler reduction}
%%%%%%%%%%%%%%%%%%%%%%%%%%%%%%%%%%%%%%%%%%%%%%%%%%%%%%%%%%%%

Let us start by recalling some well known definitions.

%%%%%%%%%%%%%%%%%%%%%%%%%%%%%%%%%%%%%%%%%%%%%%%%%%%%%%%%%%%%
\Definition
Let $G$ be a compact Lie group acting on a hyperk\"ahler
manifold $M$ by hyperk\"ahler isometries. {\bf A 
hyperk\"ahler moment map} is a smooth map $\mu: M\to\g^*\otimes\R^3$ such that:
\begin{description}
\item[(1)] $\mu$ is $G$-equivariant, i.e. $\mu(g\cdot m)={\rm Ad}_{g^{-1}}^*\mu(m)$;
\item[(2)] $\langle d\mu_i(v),\xi \rangle = \omega_i(\xi^*,v)$, for every $v\in TM$, $\xi\in\g$ and $i=1,2,3$,
where $\mu_i$ denotes one of the $3$ components of $\mu$, $\omega_i$ is one the the K\"ahler forms associated with the hyperk\"ahler structure, and $\xi^*$ is the vector field generated by $\xi$.
\end{description}
\ed

%%%%%%%%%%%%%%%%%%%%%%%%%%%%%%%%%%%%%%%%%%%%%%%%%%%%%%%%%%%%
\Definition
Let $\xi_i\in\g^*$ ($i=1,2,3$) be such that ${\rm Ad}_{g}^*\xi_i=\xi_i$, so that $G$ acts on $\mu^{-1}(\xi_1,\xi_2,\xi_3)$; suppose that this action is free.
The quotient manifold $M\3 G :=
\mu^{-1}(\xi_1,\xi_2,\xi_3)/G$ is called the {\bf
  hyperk\"ahler quotient} of $M$.
\ed

%%%%%%%%%%%%%%%%%%%%%%%%%%%%%%%%%%%%%%%%%%%%%%%%%%%%%%%%%%%%
\Theorem
Let $M$ be a hyperk\"ahler manifold, and $G$ a compact
Lie group acting on $M$ by hyperk\"ahler automorphisms,
and admitting a hyperk\"ahler moment map.
Then the hyperk\"ahler quotient $M\3 G$ is 
equipped with a natural hyperk\"ahler structure. 
\et
{\bf Proof:} See \cite{_HKLR_}, \cite[Theorem 3.35]{N} \endproof

%%%%%%%%%%%%%%%%%%%%%%%%%%%%%%%%%%%%%%%%%%%%%%%%%%%%%%%%%%%%
\subsection{Trisymplectic reduction on the space of twistor sections}
%%%%%%%%%%%%%%%%%%%%%%%%%%%%%%%%%%%%%%%%%%%%%%%%%%%%%%%%%%%%

%%%%%%%%%%%%%%%%%%%%%%%%%%%%%%%%%%%%%%%%%%%%%%%%%%%%%%%%%%%%

Let $M$ be a hyperk\"ahler manifold, 
$L\in \C \p1$ an induced complex structure
and $\ev_L:\; \Sec(M)\arrow (M,L)$
the corresponding evaluation map, 
mapping a section $s:\; \C \p1\arrow \Tw(M)$
to $s(L) \in (M,L) \subset \Tw(M)$.
Consider the holomorphic form
$\Omega_L\in \Omega^{2,0}(M,L)$ 
constructed from a hyperk\"ahler structure as in 
\eqref{_holo_symple_on_hk_Equation_}.
Denote by $\bOmega$ the space of holomorphic
forms on $\Sec(M)$ generated by $\ev_L^*(\Omega_L)$
for all $L\in \C \p1$.

%%%%%%%%%%%%%%%%%%%%%%%%%%%%%%%%%%%%%%%%%%%%%%%%%%%%%%%%%%%%
\Claim\label{_trisy_on_Sec_Claim_}
$\bOmega$ is a trisymplectic structure on the space $\Sec_0(M)$
of regular twistor sections. It generates the standard holomorphic
$SL(2)$-web, constructed in Proposition \ref{web_on_Sec}.
\ec

{\bf Proof:} Consider the bundle $\calo(2)$ on $\C \p1$,
and let $\pi^* \calo(2)$ be its lift to the twistor
space $\Tw(M)\stackrel \pi \arrow \C \p1$.
Denote by $\Omega_{\pi}^2 \Tw(M)$ the sheaf
of fiberwise 2-forms on $\Tw(M)$. 
The bundle $\Omega_{\pi}^2 \Tw(M)$ can be obtained
as a quotient 
\[ 
 \Omega_{\pi}^2 \Tw(M):= 
\frac{\Omega^2 \Tw(M)}{\pi^*\Omega^1 \C \p1\wedge \Omega^1 \Tw(M)}.
\]
It is well known (see e.g. \cite{_HKLR_}), 
that the fiberwise symplectic structure depends
on $t\in \C \p1$ holomorphically, and, moreover,
$\Tw(M)$ is equipped with a holomorphic 2-form 
$\Omega_{tw}\in \pi^*\calo(2)\otimes \Omega_{\pi}^2 \Tw(M)$
inducing the usual holomorphic symplectic forms on the fibers, see 
\cite[Theorem 3.3(iii)]{_HKLR_}.

Given $S\in \Sec(M)$, the tangent space 
$T_s\Sec(M)$ is identified with the space of
global sections of a bundle $T_{\pi}\Tw(M)\restrict S$.
Therefore, any vertical 2-form 
$\Omega_1 \in \Omega_{\pi}^2 \Tw(M)\otimes \pi^*\calo(i)$
defines a holomorphic 2-form on $\Sec_0(M)$ with values in
the space of global sections $\Gamma(\C \p1, \calo(i))$.

Denote by $A$ the space $\Gamma(\C \p1, \calo(2))$.
A fiberwise holomorphic $\calo(2)$-valued 2-form 
gives a 2-form on $NS$, for each $S\in \Sec(M)$,
with values in $A$. Therefore,
for each $\alpha \in A^*$, one obtains a 2-form
$\Omega_{tw}(\alpha)$ on $\Sec(M)$ as explained above.
Let $\bOmega$ be a 3-dimensional space
generated by $\Omega_{tw}(\alpha)$ for all $\alpha\in A^*$.

Consider a map $\epsilon_L:\; A \arrow \calo(2)\restrict L\cong \C$
evaluating $\gamma\in \Gamma(\calo(2))$ at a point $L\in \C \p1$.
By definition, the 2-form $\Omega_{tw}(\epsilon_L)$ is 
proportional to $\ev^*_L\Omega_L$. Therefore,
$\bOmega$ contains  $\ev^*_L\Omega_L$ for all
$L\in \C \p1$.

Counting parameters, we obtain that any element 
$x\in A^*$ is a sum of two evaluation maps: 
$x=a\epsilon_{L_1}+ b \epsilon_{L_2}$. When
$a,b\neq 0$ and $L_1, L_2$ are distinct, the
corresponding 2-form $a\ev^*_{L_1}\Omega_{L_1} + b\ev^*_{L_2}\Omega_{L_2}$
is clearly non-degenerate on $\Sec_0(M)$.
Indeed, the map 
\[ \Sec_0(M)\xlongrightarrow{\ev_{L_1}\times\ev_{L_2}}
   (M,L_1) \times (M,L_2)
\]
is etale, and any linear combination $a\Omega_{L_1}+b\Omega_{L_2}$
with non-zero $a,b$ is nondegenerate on $(M,L_1) \times (M,L_2)$.
When either $a$ or $b$ vanish, the corresponding form
(if non-zero) is proportional to $\ev^*_{L_i}\Omega_{L_i}$,
hence its rank is $\dim M=\frac 1 2 \dim \Sec(M)$.

Finally, note that the degenerate forms on $\bOmega$ are the ones pulled back via the evaluation maps; it follows that they form a non-degenerate quadric.

We have thus shown that $\bOmega$ is a trisymplectic
structure. Clearly, the 
annihilators of  $\ev_L^*(\Omega_L)$
form the standard 3-web on $\Sec(M)$.
Therefore, the trisymplectic structure
$\bOmega$ generates the standard $SL(2)$-web, described in Section \ref{_ratcurves_Subsection_} above.
 \endproof

\hfill

Now let $G$ be a compact Lie group acting on $M$ by
hyperk\"ahler isometries; assume that a hyperk\"ahler
moment map for the action of $G$ on $M$ exists. Let
$\Sec_0(M)$ be the space of regular twistor sections,
considered with the induced $SL(2)$-web and trisymplectic
structure. The previous Claim immediately implies the
following Proposition.

%%%%%%%%%%%%%%%%%%%%%%%%%%%%%%%%%%%%%%%%%%%%%%%%%%%%%%%%
\Proposition\label{_trisymple_on_Sec_Proposition_}
Consider a hyperk\"ahler manifold $M$ equipped with
an action of a group $G$ and a hyperk\"ahler moment map
(tacitly assumed to be equivariant).
Let $\bOmega$ be the trisymplectic structure on the space $\Sec_0(M)$
of regular twistor sections constructed in Claim \ref{_trisy_on_Sec_Claim_}.
Given any $L\in \C \p1$, let $\mu_L:\; (M,L)\arrow \g^*\otimes_\R \C$
denote the corresponding holomorphic moment map, 
obtained from the hyperk\"ahler moment map,
and consider the composition 
\[ \bMu_L:= \mu_L\circ \ev_L:\; \Sec(M)\arrow \g^*\otimes_\R \C.
\]
Then 
\[
  \bMu_\C:= \bMu_I \oplus \bMu_J \oplus \bMu_K:\; 
  \Sec_0(M)\arrow \g^*\otimes_\R\C^3
\]
is a trisymplectic moment map on $\Sec_0(M)$,
for an appropriate identification $\C^3 \cong \bOmega$.
\ep 

{\bf Proof:} Clearly, $\bMu_L$ is a moment map  
for the action of $G$ on $\Sec_0(M)$ associated
with a degenerate holomorphic 2-form $\ev_L^*(\Omega_L)$.
Indeed, for any $g\in \g=\Lie(G)$, one has
$d\mu_L(G)=\Omega_L\cntrct g$, because $\mu_L$
is a moment map form $G$ acting on $(M,L)$.
Then $d\bMu_L(g) = (\ev_L^* \Omega_L)\cntrct g$.

However, by Claim \ref{_trisy_on_Sec_Claim_},
$\bOmega=\ev_I^*\Omega_I \oplus \ev_J^*\Omega_J \oplus\ev_K^*\Omega_K$,
hence the moment map $\bMu$ associated with $\bOmega$ 
is expressed as an appropriate linear combination of
$\bMu_I, \bMu_J, \bMu_K$.
\endproof

%%%%%%%%%%%%%%%%%%%%%%%%%%%%%%%%%%%%%%%%%%%%%%%%%%%%%%%%%%%%
\subsection{Trihyperk\"ahler reduction on the space of twistor sections}
\label{_trihy_red_def_Subsection_}
%%%%%%%%%%%%%%%%%%%%%%%%%%%%%%%%%%%%%%%%%%%%%%%%%%%%%%%%%%%%

Let $\Tw(M)=M \times \C \p1$ be the twistor space of the
hyperk\"ahler manifold $M$,
considered as a Riemannian manifold with its product
metric. We normalize the Fubini-Study metric on the second
component of $\Tw(M)=M\times \C \p1$ in such a way that
$\int_{\C \p1}\Vol_{\C \p1}$ of the Riemannian volume form is $1$.

%%%%%%%%%%%%%%%%%%%%%%%%%%%%%%%%%%%%%%%%%%%%%%%%%%%%%%%%
\Claim\label{_volume_psh_Claim_}
Let $\phi$ be the area function $\Sec(M) \stackrel \phi \arrow \R^{>0}$
mapping a curve $S\in \Sec(M)$ to its its Riemannian volume
$\int_S\Vol_S$.  Then $\phi$ is a K\"ahler potential, 
that is, $dd^c\phi$ is a K\"ahler form on $\Sec(M)$,
where $d^c$ is the usual twisted differential, $d^c:=-IdI$.
\ec

{\bf Proof:} See \cite[Proposition 8.15]{_NHYM_}. \endproof

\hfill

Claim \ref{_volume_psh_Claim_} leads to the following Proposition.

%%%%%%%%%%%%%%%%%%%%%%%%%%%%%%%%%%%%%%%%%%%%%%%%%%%%%%%%
\Proposition\label{_ave_mome_Proposition_}
Assume that $G$  is a compact Lie group 
acting on $M$ by hyperk\"ahler automorphisms,
and admitting a hyperk\"ahler moment map. 
Consider the corresponding action of $G$ on
$\Sec_0(M)$, and let $\omega_\Sec=dd^c\phi$ be the K\"ahler form 
on $\Sec_0(M)$ constructed in Claim \ref{_volume_psh_Claim_}.
Then the corresponding moment map can be written as
\[
\bMu_\R(x):= \Av\limits_{L\in \C \p1} \mu^\R_L(x),
\]
where $\Av\limits_{L\in \C \p1}$ denotes the operation
of taking average over $\C \p1$, and $\mu^\R_L:\; (M,L)\arrow \g^*$
is the K\"ahler moment map associated with the action of $G$ on $(M,L)$. 
\ep

{\bf Proof:} Let $(X,I,\omega)$ be a K\"ahler manifold, $\phi$ a K\"ahler
potential on $X$, and $G$ a real Lie group preserving
$\phi$ and acting on $X$ holomorphically. Then an
equivariant moment map can be written as 
\begin{equation} \label{_moment_through_pote_Equation_}
\mu(g)=-\Lie_{I(g)}\phi,
\end{equation}
where $g\in \Lie(G)$ is an element of the Lie algebra.
Indeed, $\omega= dd^c \phi$, hence 
\[ 
  \Lie_{I(g)}\phi= d\phi\cntrct (I(g))= (d^c \phi)\cntrct g,
\]
where $\cntrct$ denotes a contraction of a differential form
with a vector field, and 
\[ 
  d\Lie_{I(g)}\phi=d((d^c \phi)\cntrct g)= \Lie_g(d^c\phi)- (dd^c\phi)\cntrct g
= -\omega\cntrct g
\]
by Cartan's formula.
Applying this argument to $X=\Sec(M)$ and $\phi=\Area(S)$, we obtain 
that $\bMu_\R(S)(g)$ is a Lie derivative of $\phi$ along $I(g)$. 

To prove that $\bMu_\R(g)$
 is equal to an average of the moment maps $\mu_L^\R(g)$, we
notice that (as follows from \cite{_NHYM_}, (8.12) and Lemma 4.4),
for any fiberwise tangent vectors $x,y\in T_{\pi} \Tw(M)$, one has
\[
 dd^c \phi(x,Iy)=\int_S(x,y)_H\Vol_{\C \p1},
\]
where $\Vol_{\C \p1}$ is the 
appropriately normalized volume form,
and $(\cdot, \cdot)_H$ the standard
Riemannian metric on $\Tw(M)=M \times S^2$.
Taking $g=y$, we obtain
\[
d (\bMu_\R g)(x)=\int_S(x,g)_H\Vol_{\C \p1}=
  \int_S d(\mu_L^\R g)(x)\Vol_{\C \p1}.
\]
The last formula is a derivative of 
an average of $d\mu^\R_L(g)$ over $L\in \C \p1$.
\endproof

\hfill

{}From Proposition \ref{_ave_mome_Proposition_} it
is is apparent that a trisymplectic quotient
of the space $\Sec_0(M)$ can be obtained using the
symplectic reduction associated with the
real moment map $\bMu_\R$. This procedure
is called {\bf the trihyperk\"ahler reduction}
of $\Sec_0(M)$.

%Let $M$ be a hyperk\"ahler manifold, and $G$ a real
%Lie group acting on $M$ by hyperk\"ahler automorphisms.
%Assume that the hyperk\"ahler moment map $\mu:\; M \arrow \g^* \otimes_\R \R^3$
%is defined. For a given $L\in \C \p1$, consider the corresponding
%real moment map $\mu_L^\R:\; M \arrow \g^*$, which can be written
%as a component of the hyperk\"ahler moment map. Consider the
%trisymplectic moment map
%$\bMu_\C:\; \Sec_0(M)\arrow \g^*\otimes_\R\C^3$
%(Proposition \ref{_trisymple_on_Sec_Proposition_}). 
%Consider also the map $\bMu_\R:\; M \arrow \g^*$, defined
%as 
%\[
%\bMu_\R(x):= \Av\limits_{L\in \C \p1} \mu^\R_L(x),
%\]
%where $\Av\limits_{L\in \C \p1}$ denotes the operation
%of taking average over $\C \p1$. 

%%%%%%%%%%%%%%%%%%%%%%%%%%%%%%%%%%%%%%%%%%%%%%%%%%%%%%%%
\Definition\label{_trihype_mome_map_Definition_}
The map $\bMu:= \bMu_\R \oplus \bMu_\C :\; \Sec_0(M) \arrow \g^*\otimes \R^7$, 
where $\bMu_\C$ is the trisymplectic moment map constructed in Proposition \ref{_trisymple_on_Sec_Proposition_}
and $\bMu_\R$ is the K\"ahler moment map constructed in Proposition \ref{_ave_mome_Proposition_},
is called {\bf the trihyperk\"ahler moment map} on $\Sec_0(M)$.
\ed

%%%%%%%%%%%%%%%%%%%%%%%%%%%%%%%%%%%%%%%%%%%%%%%%%%%%%%%%
\Definition\label{_trihype_re_Definition_}
Let $c\in \g^* \otimes \R^7$ be a $G$-invariant vector. 
Consider the space $\Sec_0(M)$ of the regular twistor 
sections. Then the 
quotient $\Sec_0(M)\4 G:= \bMu^{-1}(c)/G$ is called
{\bf the trihyperk\"ahler reduction} of $\Sec_0(M)$.
The space $\bMu^{-1}(c)/G$ is naturally identified with
the K\"ahler quotient $\bMu_\C^{-1}(c)\2 G$, hence
the space $\Sec_0(M)\4 G$ is a complex manifold.
\ed

%%%%%%%%%%%%%%%%%%%%%%%%%%%%%%%%%%%%%%%%%%%%%%%%%%%%%%%%
\Remark\label{rem_coincidence}
Note that the trihyperk\"ahler reduction $\bMu^{-1}(c)/G$
of $\Sec_0(M)$ coincides with the trisymplectic quotient 
$\bMu_\C^{-1}(c)/G_\C$, provided this last quotient is
well-defined, i.e. all $G_\C$-orbits are closed, and the
orbit space is Hausdorff and equipped with a complex
structure compatible with one on $\bMu_\C^{-1}(c)$.
Here the action of $G_\C$ on $\bMu_\C^{-1}(c)$ is obtained by
complexifying the action of $G$ on the complex manifold
$\bMu_\C^{-1}(c)$. Indeed, $(\mu_\C\oplus\mu_\R)^{-1}(c)/G$ is
precisely the space of stable
$G_\C$-orbits in $\mu_\C^{-1}(c)$.
\er

It follows from Theorem \ref{_trisy_quotient_trisy_Theorem_} 
that $\Sec_0(M)\4 G$ is equipped with an $SL(2)$-web generated by a
natural trisymplectic structure $\bOmega$, provided that 
the image of $\g=\Lie(G)$ in
$TM$ is non-degenerate at any point, in the sense of
Definition \ref{_non_dege_trisy_Definition_}.

We are finally ready to state the main result of this paper.

%%%%%%%%%%%%%%%%%%%%%%%%%%%%%%%%%%%%%%%%%%%%%%%%%%%%%%%%
\Theorem\label{_trihk_red_equal_on_sec_Theorem_}
Let $M$ be flat hyperk\"ahler manifold, and $G$ a compact
Lie group acting on $M$ by hyperk\"ahler automorphisms.
Suppose that a hyperk\"ahler moment map exists, and the
hyperk\"ahler quotient $M\3 G$ is smooth.
Then there exists an open embedding
$\Sec_0(M)\4 G\stackrel \Psi\arrow \Sec_0(M\3 G)$,
which is compatible with the trisymplectic
structures on $\Sec_0(M)\4 G$ and $\Sec_0(M\3 G)$.
\et

In particular, it follows that if $M$ is a flat hyperk\"ahler manifold, then the trihyperk\"ahler reduction of $\Sec_0(M)$ is a smooth
trisymplectic manifold whose dimension is twice that of the hyperk\"ahler quotient of $M$.

The flatness condition is mostly a technical one, but it
will suffice for our main goal, which is a description of
the moduli space of instanton bundles
on $\C\p3$ (see Section \ref{_appl_} below).

We do believe that the conclusions of Theorem \ref{_trihk_red_equal_on_sec_Theorem_}
should hold without such a condition. The crucial point is Proposition \ref{_real_mome_on_section_Proposition_} below, for which we could not
find a proof without assuming flatness; other parts of our proof, which will be completed at the end of Subsection \ref{_principal_and_red_Subsection_},
do work without it.

%%%%%%%%%%%%%%%%%%%%%%%%%%%%%%%%%%%%%%%%%%%%%%%%%%%%%%%%%%%%

%\section[Holomorphic moment map on the twistor space]%
%{Holomorphic moment map\\ on the twistor space}

%%%%%%%%%%%%%%%%%%%%%%%%%%%%%%%%%%%%%%%%%%%%%%%%%%%%%%%%%%%%

%%%%%%%%%%%%%%%%%%%%%%%%%%%%%%%%%%%%%%%%%%%%%%%%%%%%%%%%%%%%
\section{Moment map on twistor sections}
\label{_mome_map_on_sec_Subsection_}
%%%%%%%%%%%%%%%%%%%%%%%%%%%%%%%%%%%%%%%%%%%%%%%%%%%%%%%%%%%%

In this Section, we let $M$ be a \emph{flat} hyperk\"ahler
manifold. More precisely, let $M$ be an open subset of a
quaternionic vector space $V$, equipped with a flat
metric; completeness of the metric is not relevant.
Thus $\Tw(M)$ is isomorphic to the corresponding open subset of $\Tw(V)=V\otimes\op1(1)$, and $\Sec_0(M)=\Sec(M)$ is the open subset of
$\Sec(V)=V\otimes_\C \Gamma(\op1(1))\simeq V \otimes_\R \C^2$ consisting of those sections of $V\otimes\op1(1)$ that take values in $M\subset V$.

More precisely, let $[z:w]$ be a choice of homogeneous coordinates on $\C\p1$, so that $\Gamma(\op1(1)) \simeq \C z \oplus\C w$. A section
$\sigma\in Sec_0(M)$ will of the the form $\sigma(z,w)=zX_1+wX_2$ such that $\sigma(z,w)\in M$ for every $[z:w]\in\C\p1$.

Let $G$ be a compact Lie group acting on $M$ by hyperk\"ahler automorphisms, with
$\mu:M\to \g^*\otimes \langle I,J,K\rangle$
being the corresponding hyperk\"ahler moment map; let 
$\mu_I^\R$, $\mu_J^\R$, $\mu_K^\R$ denote its components.
By definition, these components are the real moment maps
associated with the symplectic
forms $\omega_I, \omega_J, \omega_K$, respectively. Given
a complex structure $L=aI +bJ + cK$, $a^2+b^2+c^2=1$,
we denote by $\mu_L^\R$ the corresponding real moment map, 
\begin{equation}\label{_real_mome_Equation_}
\mu_L^\R=a \mu_I^\R+ b\mu_J^\R+ c\mu_K^\R.
\end{equation}

The components $\mu_I, \mu_J, \mu_K$ of the hyperk\"ahler moment map can be regarded as real-valued, quadratic polynomials on $V$.
The corresponding complex linear polynomial functions $\bMu_I, \bMu_J, \bMu_K$ 
generate the trisymplectic moment map for $\Sec_0(M)$.
Consider the decomposition $V \otimes_\R \C^2=V^{1,0}_I\oplus V^{0,1}_I$,
where $I\in \End V$ acts on $V^{1,0}_I\subset V\otimes_\R \C$ 
as $\1$ and on  $V^{0,1}_I$  as $-\1$. We may regard the
trisymplectic moment map $\bMu_\C:\Sec_0(M)\arrow \g^*\otimes \C^3$
as a quadratic form $Q$ on $\Sec_0(M)\simeq V^{1,0}_I\oplus V^{0,1}_I$,  
and express it  as a sum of three components,
$$ Q^{2,0}:\; V^{1,0}_I\otimes V^{1,0}_I \arrow
\g^*\otimes \C, \ \ \ \ 
Q^{1,1}:\; V^{1,0}_I\otimes V^{0,1}_I \arrow \g^*\otimes \C $$
$$ {\rm and}~~ Q^{0,2}:V^{0,1}_I\otimes V^{0,1}_I \arrow \g^*\otimes \C. $$

For each $L\in \C \p1$, let $\mu_L^\R$ be 
the real moment map, depending on $L\in \C \p1$
as in \eqref{_real_mome_Equation_}, and consider the evaluation map
$\Sec_0(M)\stackrel{\ev_L}\arrow (M,L)$ 
(see Claim \ref{_trisy_on_Sec_Claim_} for definition). Let also
$\bMu_L^\R:= \ev_L^* \mu_L^\R$ be the pullback of
$\mu_L^\R$ to $\Sec_0(M)$.

{}From Proposition \ref{_trisymple_on_Sec_Proposition_},
the following description of the moment maps on $\Sec(M)$
can be obtained. This result will be used later
on in the proof of Theorem \ref{_trihk_red_equal_on_sec_Theorem_}.

%%%%%%%%%%%%%%%%%%%%%%%%%%%%%%%%%%%%%%%%%%%%%%%%%%%%
\Proposition\label{_real_mome_on_section_Proposition_}
Let $G$ be a real Lie group acting on a flat hyperk\"ahler
manifold $M$ by hyperk\"ahler isometries, 
$\Sec(M)\stackrel{\bMu_\C}\arrow \g^*\otimes \C^3$
the corresponding trisymplectic moment map. We 
consider the real moment map 
$\mu_L^\R$ as a $\g^*$-valued function on  $\Tw(M)=M\times \C \p1$.
Let $S\in \Sec(M)$ be a point which satisfies $\bMu_\C(S)=0$.  
Then $\mu^\R_L\restrict S$ is constant.
\ep

\begin{proof} 
We must show that for each $S\in \Sec(M)$ satisfying $\bMu_\C(S)=0$,
one has $\frac d{dL} \bMu^\R_L(S)=0$.

We express $S\in V \otimes_\R \C$ as $S=s^{1,0}_L+s^{0,1}_L$,
with $s^{1,0}_L\in V^{1,0}_L$ and $s^{0,1}_L\in V^{0,1}_L$
Then 
\begin{equation}\label{_bMu^R_via_Q^1,1_Equation_}
  \bMu^\R_L(S)=Q^{1,1}_L(s_L^{1,0}, \overline{s_L^{1,0}})
\end{equation}
where $Q^{1,1}_L$ denotes the (1,1)-component of $\bMu^\C$
taken with respect to $L$. This clear, because 
$Q^{1,1}$ is obtained by complexifying $\mu_L^\R$
(this is an $L$-invariant part of the hyperk\"ahler 
moment map).

For an ease of differentiation, we rewrite
\eqref{_bMu^R_via_Q^1,1_Equation_} as
\[
\bMu^\R_L(S)=Q(s_L^{1,0}, \overline{s_L^{1,0}})=
\Re(Q(s_L^{1,0}, s_L^{1,0})).
\]
This is possible, because $s_L^{1,0}\in V_L^{1,0}$
and $\overline{s_L^{1,0}}\in V_L^{0,1}$, hence $Q^{1,1}_L$
is the only component of $Q$ which is non-trivial on
$(s_L^{1,0}, \overline{s_L^{1,0}})$. Then
\begin{equation}\label{_diffe_bMu^R_L_Equation_}
\frac{d}{dL} \bMu^\R_L(S)\restrict{L=I}=
\Re\left[Q\left(s_I^{1,0}, \frac{ds_L^{1,0}}{dL}\restrict{L=I}\right)\right].
\end{equation}
However, $\frac{ds_L^{1,0}}{dL}\restrict{L=I}$ is clearly proportional to
$s_I^{0,1}$ (the coefficient of proportionality
depends on the choice of parametrization 
on $\C \p1\ni L$), hence \eqref{_diffe_bMu^R_L_Equation_}
gives
\[
\frac{d}{dL} \bMu^\R_L(S)\restrict{L=I}=\lambda
\Re\left[ Q(s_I^{1,0}, s_I^{0,1})\right]
\]
and this quantity vanishes, because 
\[ 
Q(s_L^{1,0}, s_L^{0,1})=Q^{1,1}(S) = \mu^\C_L(S).
\]
\end{proof}

%%%%%%%%%%%%%%%%%%%%%%%%%%%%%%%%%%%%%%%%%%%%%%%%%%%%%%%%%%%%

\section{Trisymplectic reduction and hyperk\"ahler reduction}

%%%%%%%%%%%%%%%%%%%%%%%%%%%%%%%%%%%%%%%%%%%%%%%%%%%%%%%%%%%%

%%%%%%%%%%%%%%%%%%%%%%%%%%%%%%%%%%%%%%%%%%%%%%%%%%%%%%%%%%%%
\subsection{The tautological map $\tau:\; \Sec_0(M)\4 G \arrow \Sec(M\3 G)$} \label{tau-map}
%%%%%%%%%%%%%%%%%%%%%%%%%%%%%%%%%%%%%%%%%%%%%%%%%%%%%%%%%%%%

Let $M$ be a hyperk\"ahler manifold, and $G$ a compact Lie group
acting on $M$ by hyperk\"ahler isometries, and admitting a
hyperk\"ahler moment map. A point in $\Sec_0(M)\4 G$ is 
represented by a section $S\in \Sec_0(M)$ which satisfies
$\bMu_\C(S)=0$ and $\bMu_\R(S)=0$. The first condition,
by Proposition \ref{_trisymple_on_Sec_Proposition_},
implies that for each $L\in \C \p1$, the corresponding
point $S(L)\in (M,L)$ belongs to the zero set of the
holomorphic symplectic map $\mu_L^\C:\; (M,L)\arrow \g^* \otimes_\R \C$.
Using the evaluation map defined in Claim \ref{_trisy_on_Sec_Claim_},
this is written as
\[
\mu_L^\C(\ev_L(S))=0.
\]
By Proposition \ref{_real_mome_on_section_Proposition_},
the real moment map $\mu^\R_L$ is constant on $S$:
\[
\mu_L^\C(\ev_L(S))=const.
\]
By Proposition \ref{_ave_mome_Proposition_}, 
the real part of the trihyperk\"ahler moment map
$\bMu_\R(S)$ is an average of $\mu_L^\C(\ev_L(S))$
taken over all $L\in \C \p1$. Therefore, $\bMu_\C(S)=0$ implies
\[
\bMu_\R(S)=0\Leftrightarrow \mu_L^\C(\ev_L(S))=0 \ \  \forall L\in \C \p1.
\]
We obtain that for each $S\in \Sec_0(M)$ which satisfies
$\bMu_\R(S)=0, \bMu_\C(S)=0$, and each $L\in \C \p1$, one has
\begin{equation}\label{_hk_mome_from_trihk_mome_Equation_}
\mu_L^\C(x)=0, \mu_L^\R(x)=0, 
\end{equation}
where $x= \ev_L(S)$.
A point $x\in (M,L)$ satisfying 
\eqref{_hk_mome_from_trihk_mome_Equation_}
belongs to the zero set of the hyperk\"ahler moment map
$\mu:\; M \arrow \g^* \otimes \R^3$. Taking a quotient over $G$,
we obtain a map $S/G:\; \C \p1 \arrow \Tw(M\3 G)$,
because $M\3 G$ is a quotient of $\mu^{-1}(0)$ by $G$.
This gives a map $\tau:\; \Sec_0(M)\4 G \arrow \Sec(M\3 G)$
which is called {\bf a tautological map}.
Note that $\Sec_0(M)\4 G$ has a trisymplectic
structure by Theorem \ref{_trisy_quotient_trisy_Theorem_},
outside of the set of its degenerate points (in the sense of
Definition \ref{_non_dege_trisy_Definition_}),
and $\Sec_0(M\3 G)$ is a trisymplectic manifold by
Proposition \ref{_trisymple_on_Sec_Proposition_}.

%%%%%%%%%%%%%%%%%%%%%%%%%%%%%%%%%%%%%%%%%%%%%%%%%%%%%%%%%%%%
\Proposition\label{_tauto_trisymple_Proposition_}
Let $M$ be a flat hyperk\"ahler manifold, and $G$ a compact Lie group
acting on $M$ by hyperk\"ahler isometries, and admitting a
hyperk\"ahler moment map. Consider the tautological map
\begin{equation}\label{_tautolo_map_trisy_Equation_}
  \tau:\; \Sec_0(M)\4 G \arrow \Sec(M\3 G)
\end{equation}
defined above. Then $\tau(\Sec_0(M)\4 G)$ belongs to the
set $\Sec_0(M\3 G)$ of regular twistor sections in $\Tw(M\3 G)$.
Moreover, the image of $\g$ is non-degenerate, in the sense of
Definition \ref{_non_dege_trisy_Definition_}, and 
$\tau$ is a local diffeomorphism, compatible
with the trisymplectic structure.
\ep

{\bf Proof. Step 0:} We prove that for all points $S\in
\bMu_\C^{-1}(0)$, the image of $\g$ is non-degenerate, in the sense of
Definition \ref{_non_dege_trisy_Definition_}. This is the only step
of the proof where the flatness assumption is used.
We have to show that
the image $\g_S$ of $\g$ in $T_S\Sec(M)$
is non-degenerate for all $S\in \bMu^{-1}_\C(0)$.
This is equivalent to
\begin{equation}\label{_non-degenerate_Equation_}
T_S \bMu^{-1}_\C(0)\cap \Mat(2)\g_S=\g_S.
\end{equation}
Indeed, $\g_S$ is non-degenerate
if and only if the quotient 
$T_S\Sec(M)/\Mat(2)\g_S$ is trisymplectic
(Claim \ref{_trisymple_subspa_Claim_}). 
By \eqref{_TN_as_comple_Equation_},
$T_S \bMu^{-1}_\C(0)$ is  an 
orthogonal complement of $I\g_S+J\g_S+K\g_S$ with respect
to the holomorphic Riemannian form $B$ associated with
the trisymplectic structure, where $I,J,K$
is some quaternionic basis in $\Mat(2)$.
If $\g_S$ is non-degenerate, the orthogonal
complement of $\Mat(2)\g_S$ is isomorphic
to $T_S \bMu^{-1}_\C(0)/\g_S$, which
gives \eqref{_non-degenerate_Equation_}.
Conversely, if \eqref{_non-degenerate_Equation_}
holds, the orthogonal complement
of  $I\g_S+J\g_S+K\g_S$ does not
intersect $I\g_S+J\g_S+K\g_S$,
hence the restriction of $B$ to 
$I\g_S+J\g_S+K\g_S$ is non-degenerate.
Therefore, non-degeneracy of $\g_S$ is implied by
Remark \ref{_non_dege_form_Remark_}.

Now, let $S\in \Sec_0(M)$ be a twistor section
which satisfies $\bMu_\R(S)=\bMu_\C(S)=0$.
By Proposition \ref{_real_mome_on_section_Proposition_},
 for each $L\in \C P^1$, the corresponding 
point $(L,S_L)$ of $S$ satisfies $\mu_{hk}(S_L)=0$,
where $\mu_{hk}\; M \arrow \R^3 \otimes \g^*$ 
denotes the hyperk\"ahler moment map.
Let $g\in \Mat(2)\g_S\cap T_S \bMu^{-1}_\C(0)$ be a vector
obtained as a linear combination $\sum H_i g_i$,
with $g_i \in \g_S$ and $H_i\in \Mat(2,\C)$.
At each point $(L,S_L)\in S$, $g$ is evaluated
to a linear combination $\sum H_i^L g_i^L$ with quaternionic
coefficients, tangent to $\mu_{hk}^{-1}(0)$.
However, a quaternionic linear combination
of this form can be tangent to $\mu_{hk}^{-1}(0)$
only if all $H_i^L$ are real, because for each
hyperk\"ahler manifold $Z$ one has a decomposition
$T_x(\mu_{hk}^{-1}(0))\oplus I\g\oplus J\g\oplus K\g=TxZ$.
We have proved that any $g\in \Mat(2)\g_S\cap T_S \bMu^{-1}_\C(0)$ belongs to 
the image of $\g$ at each point $(L,S_L)\in S$.
This proves \eqref{_non-degenerate_Equation_},
hence, non-degeneracy of $\g_S$.

{\bf Step 1:} We prove that the image 
$\tau(\Sec_0(M)\4 G)$ belongs to $\Sec_0(M\3 G)\subset \Sec(M\3 G)$.
Given $S \in \Sec_0(M)\4 G$, consider its image $\tau(S)$
as a curve in $\Tw(M\3 G)$, and let $N(\tau(S))$ be its normal 
bundle. Denote by $\tilde S\in \Sec_0(M)$ the twistor section
which satisfies $\bMu_\R(\tilde S)=0, \bMu_\C(\tilde S)=0$
and gives $S$ after taking a quotient.
Then 
\begin{equation}\label{_normal_tau(S)_Equation_} 
  N(\tau(S))\restrict L = 
\frac{T_{\ev_L(\tilde S)}(M,L)}{\langle \g + I\g + J \g + K\g\rangle}
\end{equation}
where $\ev_L:\; \Sec(M) \arrow (M,L)$ is the
standard evaluation map. 

A bundle $B \cong \bigoplus_{2n} \calo(1)$
can be constructed from a quaternionic vector space $W$
as follows. For any $L\in \C \p1$, considered as a quaternion
satisfying $L^2=-1$, one takes the complex vector space $(W,L)$
as a fiber of $B$ at $L$. Denote this bundle as $B(W)$.
Now, \eqref{_normal_tau(S)_Equation_} 
gives 
\[ 
  N(\tau(S))= \frac{N(\tilde S)}
 {B(\langle \g + I\g + J \g + K\g\rangle)}
\]
giving a quotient of $\bigoplus_{2i} \calo(1)$
by $\bigoplus_{2j} \calo(1)$, which is also a direct
sum of $\calo(1)$. Therefore, $\tau(S)$ is regular.

\hfill

{\bf Step 2:} The tautological map 
$\tau:\; \Sec_0(M)\4 G \arrow \Sec(M\3 G)$
is a local diffeomorphism. This follows from the
implicit function theorem. Indeed, let
$S\in \Sec_0(M)\4 G$ be a point associated with
$\tilde S\in \Sec_0(M)$, satisfying 
$\bMu_\R(\tilde S)=0, \bMu_\C(\tilde S)=0$
as in Step 1. Then the differential of $\tau$ is a map
\begin{equation}\label{_dtau_Equation_}
d\tau:\; \frac{\Gamma(N\tilde S)}{\Mat(2, \C)\cdot \g}\arrow
\Gamma(N\tau(S)),
\end{equation}
where $\g = \Lie(G) \subset T\Tw(M)$.
Let $N_g \tilde S$ be a sub-bundle of $N\tilde S$
spanned by the image of $\langle \g + I\g + J \g + K\g\rangle$
By Step 1, $N_gS\cong \calo(1)^k$, and, indeed, a 
subspace of $\Gamma(N\tilde S)$ generated by
$\Mat(2, \C)\cdot \g$ coincides with
$\Gamma(N_g \tilde S)$. Similarly,
$\Gamma(N\tau(S))\cong \Gamma(N\tilde S/N_g \tilde S)$.
We have shown that the map \eqref{_dtau_Equation_}
is equivalent to 
\[
\frac{\Gamma(N\tilde S)}{\Gamma(N_g \tilde S)} \arrow
\Gamma(N\tilde S/N_g \tilde S).
\]
By step 1, the bundles $N\tilde S$ and $N_g \tilde S$
are sums of several copies of $\calo(1)$, hence
this map is an isomorphism.

\hfill

{\bf Step 3:} We prove that $\tau$ is compatible with the
trisymplectic structure. The trisymplectic structure on
$\Sec(M\3 G)$ is induced by a triple of holomorphic symplectic
forms $\langle \ev_I^*(\Omega_I), \ev_J^*(\Omega_J), 
\ev_K^*(\Omega_K)\rangle$ (Claim \ref{_trisy_on_Sec_Claim_}).
{}From the construction in Theorem \ref{_trisy_quotient_trisy_Theorem_}
it is apparent that the same triple generates
the trisymplectic structure on $\Sec_0(M)\4 G$. Therefore,
$\tau$ is compatible with the trisymplectic structure.
We proved Proposition \ref{_tauto_trisymple_Proposition_}.
\endproof

%%%%%%%%%%%%%%%%%%%%%%%%%%%%%%%%%%%%%%%%%%%%%%%%%%%%%%%%%%%%
\subsection{Trihyperk\"ahler reduction and homogeneous bundles on $\C \p1$}
\label{_principal_and_red_Subsection_}
%%%%%%%%%%%%%%%%%%%%%%%%%%%%%%%%%%%%%%%%%%%%%%%%%%%%%%%%%%%%

Let $M$ be a hyperk\"ahler manifold, and $G$ a compact
Lie group acting on $M$ by hyperk\"ahler isometries,
and equipped with a hyperk\"ahler moment map. 
Consider the set $Z\subset \Tw(M)$
consisting of all points $(m,L)\in \Tw(M)$
such that the corresponding holomorphic
moment map vanishes on $m$: $\mu_L^\C(m)=0$.
By construction, $Z$ is a complex subvariety
of $\Tw(M)$. Let $G_\C$ be a complexificaion
of $G$, acting on $\Tw(M)$ in a natural way, and
$G_\C\cdot(m,L)$ its orbit. This orbit is called
{\bf stable}, if $G_\C\cdot m \subset (M,L)$
intersects the zero set of the real moment map,
\[
G_\C\cdot m\cap \left(\mu_L^\R\right)^{-1}(0)\neq \emptyset.
\]
As follows from the standard results about K\"ahler reduction,
the union $Z_0\subset Z$ of stable orbits is open in $Z$,
and the quotient $Z_0/G_\C$ is isomorphic, as a complex
manifold, to $\Tw(M\3 G)$. Consider the corresponding
quotient map, 
\begin{equation}\label{_quo_in_twi_Equation_}
P:\; Z_0\arrow Z_0/G_\C=\Tw(M\3 G).
\end{equation}
For any twistor section $S\in \Sec(M\3 G)$,
consider its preimage $P^{-1}(S)$. Clearly,
$P^{-1}(S)$ is a holomorphic homogeneous vector
bundle over $S\cong \C \p1$. We denote this
bundle by $P_S$.

%%%%%%%%%%%%%%%%%%%%%%%%%%%%%%%%%%%%%%%%%%%%%%%%%%%%%%%%%%%%%%%%%
\Proposition\label{_tau_map_and_bundles_on_CP^1_Proposition_}
Let $M$ be a flat hyperk\"ahler manifold, and $G$ a compact Lie group
acting on $M$ by hyperk\"ahler isometries, and admitting a
hyperk\"ahler moment map. Consider the tautological map
$\tau:\; \Sec_0(M)\4 G \arrow \Sec_0(M\3 G)$
constructed in Proposition 
\ref{_tauto_trisymple_Proposition_}.
Given a twistor section  $S\in \Sec_0(M\3 G)$, let
$P_S$ be a holomorphic homogeneous bundle
constructed above. Then 
\begin{description}
\item[(i)] The point $S$ lies in $\im \tau$
if and only if the bundle $P_S$ admits
a holomorphic section (this is equivalent
to $P_S$ being trivial).
\item[(ii)] The map $\tau:\; \Sec_0(M)\4 G \arrow \Sec_0(M\3 G)$
is an open embedding.
\end{description}
\ep
{\bf Proof:} A holomorphic section $S_1$ of $P_S$
can be understood as a point in $\Sec(M)$. Since 
$S_1$ lies in the union of all stable orbits, denoted 
earlier as $Z_0\subset Z\subset \Tw(M)$, the real moment map
$\mu_L^\R$ is constant on $S_1$ (Proposition 
\ref{_real_mome_on_section_Proposition_}). 
By definition of $Z_0$, for each $(z,L)\in Z_0$,
there exists $g\in G_\C$ such that $\mu_L^\R(gz)=0$.

Therefore, $\bMu_\R(gS_1)=0$ for appropriate
$g\in G_\C$. This gives $\tau(S_2)=S$, where 
$S_2\in \Sec_0(M)\4 G$ is a point corresponding to $gS_1$.
Conversely, consider a point 
$S_2\in \Sec_0(M)\4 G$, such that $\tau(S_2)=S$,
and let $S_1\in S_0(M)$ be the corresponding twistor
section. Then $S_1$ gives a section of $P_S$.
We proved Proposition \ref{_tau_map_and_bundles_on_CP^1_Proposition_}
(i). 

To prove  Proposition \ref{_tau_map_and_bundles_on_CP^1_Proposition_}
(ii), it would suffice to show the following.
Take $S\in \Sec_0(M\3 G)$, and let $S_1, S_2\in\Sec(M)$ be twistor
sections which lie in $Z_0$ and satisfy $\bMu_\R(S_i)=0$.
Then there exists $g\in G$ such that $g(S_1)=S_2$.
Indeed, $\tau^{-1}(S)$ is the set of all such $S_i$
considered up to an action of $G$. 

Let $P_S\stackrel P \arrow  S$ be the homogeneous bundle
constructed above, and ${\cal P}$ its fiber, which
is a complex manifold with transitive action of $G_\C$.
Using $S_1$, we trivialize $P_S={\cal P}\times S$ in such a way that
$S_1=\{p\}\times S$ for some $p\in {\cal P}$.
Then $S_2$ is a graph of a holomorphic map
$\C \p1 \stackrel\phi\arrow {\cal P}$; to prove
 Proposition \ref{_tau_map_and_bundles_on_CP^1_Proposition_}
(ii) it remains to show that $\phi$ is constant.

Since all points of $\left(\mu_L^\R\right)^{-1}(0)$
lie on the same orbit of $G$, the image $\phi(\C \p1)$
belongs to $G_p:= G\cdot\{p\}\subset {\cal P}$.
However, $G_p$ is a totally real subvariety
in ${\cal P}= G_\C/\St(p)$. Indeed, $G_p$
is fixed by a complex involution which exchanges
the complex structure on $G_\C$ with its opposite.
Therefore, all complex subvarieties of $G_p$
are 0-dimensional, and $\phi:\; \C \p1 \arrow G_p\subset  {\cal P}$
is constant. We finished a proof of 
Proposition \ref{_tau_map_and_bundles_on_CP^1_Proposition_}.
\endproof

\hfill

The proof of Theorem \ref{_trihk_red_equal_on_sec_Theorem_}
follows. Indeed, by Proposition 
\ref{_tauto_trisymple_Proposition_}, the tautological
map  $\tau:\; \Sec_0(M)\4 G \arrow \Sec_0(M\3 G)$
is a local diffeomorphism compatible with the
trisymplectic structures, and by 
Proposition \ref{_tau_map_and_bundles_on_CP^1_Proposition_}
it is injective.

In particular, we have:

\begin{corollary}\label{smooth_corollary}
Let $M$ be a flat hyperk\"ahler manifold equipped with the action of
a compact Lie group by hyperk\"ahler isometries and admitting a
hyperk\"ahler moment map. Suppose that the hyperk\"ahler quotient is smooth and the 
trisymplectic quotient $\bMu_\C^{-1}(c)/G_\C$ is well defined. Then the trihyperk\"ahler reduction
of $\Sec_0(M)$ is a smooth trisymplectic manifold of dimension $2\dim M$.
\end{corollary}
\endproof

%%%%%%%%%%%%%%%%%%%%%%%%%%%%%%%%%%%%%%%%%%%%%%%%%%%%%%%%%%%%%%%%%%%%%%%%%%%%%%

\section{Case study: moduli spaces of
  instantons} \label{_appl_}

%%%%%%%%%%%%%%%%%%%%%%%%%%%%%%%%%%%%%%%%%%%%%%%%%%%%%%%%%%%%%%%%%%%%%%%%%%%%%%

In this Section, we give an application of the previous geometric constructions
to the study of the moduli space of framed instanton bundles on $\C\p3$. Our goal is to establish the smoothness of the moduli space of such
objects, and show how that proves the smoothness of the
moduli space of mathematical instanton bundles on
$\C\p3$. That partially settles the long standing conjecture
in algebraic geometry mentioned at the Introduction: the moduli space of
mathematical instanton bundles on $\C\p3$ of charge $c$ is
a smooth manifold of dimension $8c-3$, cf. \cite[Conjecture 1.2]{CTT}.

%%%%%%%%%%%%%%%%%%%%%%%%%%%%%%%%%%%%%%%%%%%%%%%%%%%%%%%%%%%%%%%%%%%%%%%%
\subsection{Moduli space of framed instantons on $\R^4$} 
\label{instantons_r4}
%%%%%%%%%%%%%%%%%%%%%%%%%%%%%%%%%%%%%%%%%%%%%%%%%%%%%%%%%%%%%%%%%%%%%%%%

We begin by recalling the celebrated ADHM construction of instantons, which gives a description of the moduli space of framed instantons on $\R^4$ in terms of a finite-dimensional hyperk\"ahler quotient.

Let $V$ and $W$ be complex vector spaces of dimension $c$ and $r$, respectively.
$$ \mathbf{B}=\mathbf{B}(r,c) := {\rm End}(V)\oplus{\rm End}(V)\oplus {\rm Hom}(W,V)\oplus{\rm Hom}(V,W). $$
A point of $\mathbf{B}$ is a quadruple $X=(A,B,I,J)$ with $A,B\in{\rm End}(V)$, $I\in{\rm Hom}(W,V)$ and $J\in{\rm Hom}(V,W)$. Together with its natural complex structure, the anti-linear involution
$X\to X^*:=(B^\dagger,-A^\dagger,J^\dagger,-I^\dagger)$ provides $\mathbf{B}$ with the structure of a quaternionic vector space; in particular, $\mathbf{B}$ becomes a flat hyperk\"ahler manifold.

A quadruple $X=(A,B,I,J)$ it is said to be
\begin{enumerate}
\item[(i)] {\em stable} if there is no subspace $S\subsetneqq V$ with $A(S),B(S),I(W)\subset S$;
\item[(ii)] {\em costable} if there is no subspace $0\neq S\subset V$ with $A(S),B(S)\subset S\subset \ker J$;
\item[(iii)] {\em regular} if it is both stable and costable.
\end{enumerate}
Let $\mathbf{B}^{\rm reg}$ denote the (open) subset of regular data. The group $G=U(V)$ acts on $\mathbf{B}^{\rm reg}$ in the following way:
\begin{equation}\label{action}
g\cdot(A,B,I,J) := (gAg^{-1},gBg^{-1},gI,Jg^{-1}).
\end{equation}
It is not difficult to see that this action is free and preserves the hyperk\"ahler structure provided above. The hyperk\"ahler moment map $\mu:\mathbf{B}^{\rm reg}\to\gu(V)^*\otimes\R^3$ can then be written in the following manner, cf. \cite[Section 3.2]{N}. Using the decomposition $\R^3\simeq\C\oplus\R$ (as real vector spaces), we decompose $\mu=(\mu_\C,\mu_\R)$ with
$\mu_\C$ and $\mu_\R$ given by
\begin{equation}\label{mmC}
\mu_\C(A,B,I,J) = [A,B]+IJ ~~{\rm and}
\end{equation}
\begin{equation}\label{mmR}
\mu_\R(A,B,I,J) = [A,A^\dagger]+[B,B^\dagger] + II^\dagger - J^\dagger J .
\end{equation}
The first component $\mu_\C$ is the holomorphic moment map $\mathbf{B}\to\gl(V)^*\otimes_\R\C$ corresponding to the natural complex structure on $\mathbf{B}$.

The so-called \emph{ADHM construction}, named after Atiyah, Drinfeld, Hitchin and Manin \cite{_ADHM_},
provides a bijection between the hyperk\"ahler quotient
$\calm(r,c):=\mathbf{B}^{\rm reg}(r,c)\3 U(V)$ and the moduli space of framed
instantons on the Euclidean 4-dimensional space $\R^4$;
see \cite{D1} or \cite[Theorem 3.48]{N}, and the references therein for details. Moreover, Maciocia has shown in \cite{Mac} that such bijection is a hyperk\"ahler isometry between $\calm(r,c)$ and the moduli space of framed instantons on $\R^4$ provided with its $L^2$-metric.

Let us now consider the trisymplectic reduction of $\Sec_0(\mathbf{B}^{\rm reg})$. As noted in the first few paragraphs of Section \ref{_mome_map_on_sec_Subsection_},
$\Sec_0(\mathbf{B})=\Sec(\mathbf{B})\simeq \mathbf{B}\otimes \Gamma(\op1(1))$, and $\Sec_0(\mathbf{B}^{\rm reg})$
is the (open) subset of $\Sec_0(\mathbf{B})$ consisting of those sections $\sigma$ such that $\sigma(p)$ is regular for every $p\in\C\p1$.

\begin{definition} \label{gr-def}
A section $\sigma\in \mathbf{B}\otimes \Gamma(\op1(1))$ is \emph{globally regular} if $\sigma(p)\in\mathbf{B}$ is regular for every $p\in\C\p1$ (cf. \cite[p. 2916-7]{FJ2}, where such sections are called \emph{$\C$-regular}).
\end{definition}

To be more precise, let $[z:w]$ be homogeneous coordinates on $\C\p1$; such choice leads to identifications
\begin{equation}\label{decompositions}
\Gamma(\op1(1)) \simeq \C z \oplus \C w \simeq \C^2 ~~{\rm and}~~ \Gamma(\op1(2)) \simeq \C z^2 \oplus \C w^2\oplus \C zw \simeq \C^3.
\end{equation}
It follows that $\Sec_0(\mathbf{B}) \simeq\mathbf{B}\oplus\mathbf{B}$, so a point $\widetilde{X}\in\Sec_0(\mathbf{B})$ can regarded as a pair $(X_1,X_2)$ of ADHM data; $\widetilde{X}$ is globally regular (i.e. $\widetilde{X}\in\Sec_0(\mathbf{B}^{\rm reg})$) if any linear combination $zX_1+wX_2$ is regular.

The action (\ref{action}) of $GL(V)$ (hence also of $U(V)$) on $\mathbf{B}^{\rm reg}$ extends to \linebreak $\Sec_0(\mathbf{B}^{\rm reg})$ by acting trivially on the $\Gamma(\op1(1))$ factor, i.e. $g\cdot(X_1,X_2)=(g\cdot X_1,g\cdot X_2)$. 

Using the identification $\C^3\simeq \Gamma(\op1(2))$ above, it follows that the trisymplectic moment map
$$ \bMu_\C:\Sec_0(\mathbf{B}^{\rm reg})\to\gu(V)^*\otimes_\R\Gamma(\op1(2)) $$
constructed in Proposition \ref{_trisymple_on_Sec_Proposition_} satisfies
$\bMu_\C(\sigma)(p) = \mu_\C(\sigma(p))$ for $\sigma\in\Sec_0(\mathbf{B}^{\rm reg})$ and $p\in\C\p1$.

More precisely, let $X_1=(A_1,B_1,I_1,J_1)$ and $X_2=(A_2,B_2,I_2,J_2)$; consider the section $\sigma(z,w)=zX_1+wX_2\in\Sec_0(\mathbf{B}^{\rm reg})$. The identity $\bMu_\C(\sigma)(p) = \mu_\C(\sigma(p))$ means that $\bMu_\C(\sigma)=0$ iff $\mu_\C(zX_1+wX_2)=0$ for every $[z:w]\in\C\p1$. Note that
\begin{equation}\label{1-diml-eqns}
\mu_\C(zX_1+wX_2)=0 \Leftrightarrow 
\left\{
\begin{array}{l}
~[ A_1 , B_1 ]+I_1J_1 = 0 \\
~[ A_2 , B_2 ]+I_2J_2 = 0 \\
~[ A_1 , B_2 ] + [ A_2,B_1 ] + I_1J_2 + I_2J_1 = 0
\end{array} \right.
\end{equation}

The three equations on the right hand side of equation (\ref{1-diml-eqns}) are known as the \emph{1-dimensional ADHM equations}; they were first considered by Donaldson in \cite{D1} (cf. equations (a-c) in \cite[p. 456]{D1}) and further studied in \cite{FJ2} (cf. equations (7-9) in \cite[p. 2917]{FJ2}) and generalized in \cite[equation (3)]{J-cr}.

One can show that globally regular solutions of the
1-dimensional ADHM equations are GIT-stable with respect
to the $GL(V)$-action, see \cite[Section 3]{_Hauzer_Langer_}
and \cite[Section 2.3]{HJV}. Therefore, according to Remark \ref{rem_coincidence},
the trihyperk\"ahler quotient $\Sec_0(\mathbf{B}^{\rm reg})\4 U(V)$
is well defined and coincides with $\bMu_\C^{-1}(0)/GL(V)$.

%------------------------------------------------------------------------------
\subsection{Moduli space of framed instanton bundles on  $\C\p3$} 
\label{framed_instantons_bundles}
%------------------------------------------------------------------------------

Recall that an {\em instanton bundle} on $\C\p3$ is a locally free coherent sheaf $E$ on $\C\p3$ satisfying the following conditions
\begin{itemize}
\item $c_1(E)=0$;
\item $H^0(E(-1))=H^1(E(-2))=H^{2}(E(-2))=H^3(E(-3))=0$.
\end{itemize}
The integer $c:=c_2(E)$ is called the {\em charge} of $E$. One can show that if $E$ is an instanton bundle on $\C\p3$, then $c_3(E)=0$.

Moreover, a locally free coherent sheaf $E$ on $\C\p3$ is
said to be \emph{of trivial splitting type} if there is a line
$\ell\subset\C\p3$ such that the restriction $E|_\ell$ is the free sheaf, i.e.
$E|_\ell\simeq{\cal O}_{\ell}^{\oplus{\rm rk}E}$. A {\em framing} on $E$ at the line $\ell$ is the choice of an isomorphism $\phi:E|_\ell\to{\cal O}_{\ell}^{\oplus{\rm rk}E}$. A {\em framed bundle} (at $\ell$) is a pair $(E,\phi)$ consisting of a locally free coherent
sheaf $E$ of trivial splitting type and a framing $\phi$ at $\ell$. Two framed bundles
$(E,\phi)$ and $(E',\phi')$ are isomorphic if there exists a bundle isomorphism
$\Psi:E\to E'$ such that $\phi'=\phi\circ(\Psi|_\ell)$. Let $\calf_\ell(r,c)$ denote the moduli space of isomorphism classes of rank $r$, charge $c$ instanton bundles on $\C\mathbb{P}^{3}$ framed at a fixed
line $\ell$.

\begin{proposition}\label{_insta_same_as_trihk_Theorem_}
The moduli space $\calf_\ell(r,c)$ is naturally identified, as a complex manifold, with the trihyperk\"ahler
reduction $\Sec_0(\mathbf{B}^{\rm reg}(r,c))\4 U(V)$.
\end{proposition}

\begin{proof}
According to \cite[Theorem 4.2]{HJV} (see also \cite[Theorem 5.3]{_Hauzer_Langer_})
the quasi-projective algebraic variety $\bMu_\C^{-1}(0)/GL(V)$ is a fine moduli space
for the isomorphism classes of rank $r$ framed instantons bundles of charge $c$ on $\p3$.
It follows that $\calf_\ell(r,c)$ is naturally identified with $\bMu_\C^{-1}(0)/GL(V)$
as a complex manifold. In the last paragraph of Section \ref{instantons_r4}, we
identified the latter space with $\Sec_0(\mathbf{B}^{\rm reg}(r,c))\4 U(V)$.
\end{proof}

\hfill

We are finally in position to use Theorem \ref{_trihk_red_equal_on_sec_Theorem_} to obtain the second main result of this paper.

\begin{theorem}\label{_instanton bdls_}
The moduli space $\calf_\ell(r,c)$ of rank $r$, charge $c$ instanton bundles on $\C\mathbb{P}^{3}$ framed at a fixed line $\ell$, is a smooth trisymplectic manifold of complex dimension $4rc$.
\end{theorem}

\begin{proof}
The moduli space $\calm(r,c):=\mathbf{B}^{\rm reg}(r,c)\3
U(V)$ of framed instantons of rank $r$ and charge $c$ is
known to be a smooth, connected, hyperk\"ahler manifold of
complex dimension $2rc$; it follows that
$\Sec_0(\calm(r,c))$ is a smooth, trisymplectic manifold
of complex dimension $4rc$ (cf. Section \ref{_ratcurves_Subsection_} above).
By the ADHM construction of framed instanton bundles, a point in $\calf_\ell(r,c)$ can 
be regarded as a pair $(X_1,X_2)\in\mathbf{B}\oplus\mathbf{B}$
satisfying equations \ref{1-diml-eqns}. Therefore we have a map
\begin{equation} \label{insta_to_sec_Equation_}
\calf_\ell(r,c) \arrow \Sec(\calm(r,c))
\end{equation}
given by 
$$ (X_1,X_2) \mapsto \sigma(z,w)=zX_1+wX_2, $$
with $\sigma:\p1\to\calm(r,c)$ being a twistor section. By \cite[Theorem 3.9]{_JV:Instantons_},
this map is an isomorphism (without the condition of regularity,
which implies smoothness). 

It follows from Proposition \ref{_insta_same_as_trihk_Theorem_}
that $\calf_\ell(r,c)$ is the trihyperk\"ahler reduction of $\Sec_0(\mathbf{B}^{\rm reg}(r,c))$.
From its construction, it is clear that the map \eqref{insta_to_sec_Equation_} 
coincides with the map
\begin{equation}\label{_trihype_cano_inst_Equation_}
 \calf_\ell(r,c) =\Sec_0(\mathbf{B}^{\rm reg}(r,c))\4 U(V) \arrow
  \Sec(\mathbf{B}^{\rm reg}(r,c)\3 U(V)=\Sec(\calm(r,c))
\end{equation}
constructed in Theorem \ref{_trihk_red_equal_on_sec_Theorem_} (cf. Section \ref{tau-map});
it then follows that\eqref{_trihype_cano_inst_Equation_}
is in fact an open embedding to $\Sec_0(\calm(r,c))$.
Since \eqref{insta_to_sec_Equation_}  is 
an isomorphism, \[ \Sec_0(\calm(r,c))= \Sec(\calm(r,c)).\]
This former space is smooth, which proves smoothness
of $\calf_\ell(r,c)$.
\end{proof}

%%%%%%%%%%%%%%%%%%%%%%%%%%%%%%%%%%%%%%%%%%%%%%%%%%%%%%%%%%%%
\Remark \rm
Notice that Theorem \ref{_trihk_red_equal_on_sec_Theorem_}
in itself only shows that the space
$\calf_\ell(r,c)$, which is a trihyperk\"ahler reduction
of $\Sec_0(\mathbf{B}^{\rm reg}(r,c))$, is 
openly embedded to $\Sec_0(\calm(r,c))$.
This already proves that $\calf_\ell(r,c)$ is smooth, but
to prove that this map is an isomorphism, we use
\cite[Theorem 3.9]{_JV:Instantons_}.
\er

%%%%%%%%%%%%%%%%%%%%%%%%%%%%%%%%%%%%%%%%%%%%%%%%%%%%%%%%%%%%%%%%%%%%%%%%
\subsection{Moduli space of rank $\mathbf{2}$ instanton
  bundles on $\C\p3$} \label{instantons_bundles}
%%%%%%%%%%%%%%%%%%%%%%%%%%%%%%%%%%%%%%%%%%%%%%%%%%%%%%%%%%%%%%%%%%%%%%%%

Let us now focus on the case of rank $2$ instanton bundles, which is rather special.
Recall that a mathematical instanton bundle on $\C\p3$ is a rank $2$ stable bundle $E\to\C\p3$ with $c_1(E)=0$ and $H^1(E(-2))=0$.

\Proposition
Rank $2$ instanton bundles on $\C\p3$ are precisely mathematical instanton bundles.
\ep
\begin{proof}
if $E$ is a mathematical instanton bundle, then $H^0(E(-1))=0$ by stability. Since $\Lambda^2E=\op3$, there is a (unique up to a scalar) symplectic isomorphism between $E$ and its dual $E^*$; one can then use Serre duality to show that $H^2(E(-2))=H^3(E(-3))=0$,  thus $E$ is a rank $2$ instanton bundle.

Conversely, every instanton bundle can be presented as the
cohomology of a linear monad on $\C\p3$ \cite[Theorem 3]{J-i}.
It is then a classical fact that if $E$ is a
rank $2$ bundle obtained as the cohomology of a linear
monad on $\C\p3$ then $E$ is stable. It is then clear that
every rank $2$ instanton bundle is a  mathematical
instanton bundle.
\end{proof}

Let $\mathcal{I}(c)$ denote the moduli space of mathematical instanton bundles and $\mathcal{I}_\ell(c)$ the open subset of $\mathcal{I}(c)$ consisting of instanton bundles restricting trivially to a fixed line $\ell\subset\C\p3$.

Let also $\mathcal{G}(c)$ denote the moduli space of S-equivalence classes of semistable torsion-free sheaves $E$ of rank $2$ on $\p3$ with $c_1(E)=0$, $c_2(E)=c$ and $c_3(E)=0$; it is a projective variety. $\mathcal{I}(c)$ can be regarded as the open subset of $\mathcal{G}(c)$ consisting of those locally free sheaves satisfying $H^1(E(-1))=0$.

For any fixed line $\ell\subset\C\p3$, $\mathcal{I}(c)$ is contained in $\overline{\mathcal{I}_{\ell}(c)}$, where the closure is taken within
$\mathcal{G}(c)$. Thus $\mathcal{I}(c)$ is irreducible if and only if
there is a line $\ell$ such that $\mathcal{I}_{\ell}(c)$ is irreducible.

Using a theorem due to Grauert and M\"uhlich we can conclude
that every mathematical instanton bundle must restrict trivially at some line $\ell\subset\C\mathbb{P}^{3}$ 
(see \cite[Lemma 3.12]{_JV:Instantons_}). Therefore, $\mathcal{I}(c)$ is covered by open subsets of the form $\mathcal{I}_{\ell}(c)$, but it is not contained within any such sets, since for any nontrivial bundle over $\C\p3$ there must exist a line $\ell'$ such that the restricted sheaf $E|_{\ell'}$ is nontrivial. Thus $\mathcal{I}(c)$ and $\mathcal{I}_{\ell}(c)$ must have the same dimension, and one is nonsingular if and only if the other is as well.

We are now ready to prove the smoothness of the moduli space of mathematical instanton bundles on $\C\p3$.

\begin{theorem}
The moduli space $\mathcal{I}(c)$ of mathematical
instanton bundles on $\C\p3$ of charge $c$  is a smooth
complex manifold of dimension $8c-3$.
\end{theorem}
\begin{proof}
The forgetful map $\calf_l(2,c)\to\mathcal{I}_\ell(c)$ that takes the pair $(E,\phi)$ simply to $E$ has as fibers the the set of all possible framings at $\ell$ (up to equivalence). Since 
$E|_\ell\simeq W\otimes\mathcal{O}_\ell$ \cite[Proposition 13]{FJ2}, a choice of framing correspond to a choice of basis for the $2$-dimensional space $W$, thus all fibers of the forgetfull map are isomorphic to $SL(W)$. Since $\calf_l(2,c)$ is smooth of dimension $8c$, we conclude that $\mathcal{I}_\ell(c)$ is also smooth and its dimension is $8c-3$. The Theorem follows from our previous discussion.
\end{proof}

The irreducibility of $\mathcal{I}(c)$ for arbitrary $c$ remains an open problem; it is only known to hold for $c$ odd \cite{T} or $c=2,4$ (see \cite{CTT} and the references therein). Clearly, if $\calf_l(2,c)$ is connected, then it must be irreducible, from which one concludes that $\mathcal{I}_\ell(c)$, and hence $\mathcal{I}(c)$, are irreducible. Since  $\calf_l(2,c)$ is a quotient of the set of globally regular solutions of the 1-dimensional ADHM equations, it is actually enough to prove that the latter is connected. 

It is also worth mentioning a recent preprint of Markushevich and Tikhomirov \cite{MT}, in which the authours prove that $\mathcal{I}(c)$ is rational whenever it is irreducible. Thus, one also concludes immediately that $\calf_l(2,c)$ is also rational whenever it is irreducible.

\end{document}